\documentclass[a4paper, 12pt]{article}

\usepackage{Preambule}
\usepackage{crefprefix}
\usepackage{enuminthm}
\setenuminthmdefaultformats{\roman*)}{\roman*}
\renewcommand{\comp}{c}

\title{Filtered deformations of Lie groupoids}
\author{Paul Le Breton}
\date{}

\newtheorem{theoreme}{Theorem}[section]

\newtheorem{lemme}[theoreme]{Lemma}

\newtheorem{corollaire}[theoreme]{Corollary}

\newtheorem{proposition}[theoreme]{Proposition}

\theoremstyle{definition}
\newtheorem{definition}[theoreme]{Definition}

\theoremstyle{remark}
\newtheorem{remarque}[theoreme]{Remark}

\theoremstyle{remark}
\newtheorem{exemple}[theoreme]{Example}

\begin{document}
\maketitle

\begin{abstract}
Let $G \rightrightarrows \units G$ be a Lie groupoid, $\bundle A G \rightarrow \units G$ its Lie algebroid and $X_1, \dots, X_r$ a family of sections of $\bundle A G$ satisfying a Lie bracket generating condition of Hörmander type. We aim to build a pseudodifferential calculus allowing to study a Helffer-Nourrigat's conjecture on the groupoid $G$; in particular, we want differential operators of the form $\sum_{i = 1}^r X_i^2$ to have an invertible symbol. In this article we achieve the geometrical part of this construction by defining a "weighted" version of the deformation to the normal cone $\DNC(G, \units G) \rightrightarrows \units G \times \R_+$. Heuristically, we deform $G$ around $\units G$ with a "zoom" parameter $t \in \R_+$ by stretching $G$ by $t$ in the directions of the sections $X_i$, $t^2$ along $[X_i, X_j]$, $t^3$ along $[X_i, [X_j, X_k]]$ etc. In the case where $G = M \times M$ we recover a construction of Mohsen \cite{Mohsen24}, and when the structure is equiregular we recover a construction of van Erp-Yuncken \cite{vanErpYuncken16}.
\end{abstract}

\tableofcontents
\section*{Introduction}
\addcontentsline{toc}{section}{Introduction}
Since the introduction of pseudodifferential operators by Kohn-Nirenberg \cite{KohnNirenberg65}, many authors have built pseudodifferential calculi adapted to various singular cases; Lie groupoids happen to be a powerful and unifying tool for such constructions. Indeed, to any Lie groupoid $G \rightrightarrows \units G$ is naturally associated a pseudodifferential calculus, namely a graded family of operators together with a symbol map. This construction in full generality has been independantly developped by Nistor-Weinstein-Xu \cite{NistorWeinsteinXu} and Monthubert-Pierrot \cite{MonthubertPierrot}, formalizing the pioneering work of Connes \cite{Connes79} on foliations. Using the pair groupoid $M \times M \rightrightarrows M$, one recovers Kohn-Nirenberg's pseudodifferential operators on the manifold $M$. Using holonomy groupoids, Connes built a longitudinal calculus \cite{Connes79} on foliated manifolds; it is actually the first instance of the use of pseudodifferential calculus on a Lie groupoid. Later on Monthubert built the b-groupoid \cite{Monthubert} and recovered Melrose's b-calculus \cite{Melrose93} on manifolds with corners. Mazzeo's edge calculus \cite{Mazzeo91} and $\phi$-calculus \cite{MazzeoMelrose98} also both correspond to a groupoid, see for example \cite{DebordSkandalis19} and \cite{Rochon12}. In \cite{DebordLescureRochon15}, the authors used groupoids to build a cusp-type calculus on stratified pseudomanifolds.

In \cite{DebordSkandalis14}, Debord-Skandalis gave an alternative equivalent definition of the pseudodifferential calculus on a Lie groupoid $G \rightrightarrows \units G$. The previous approach viewed pseudodifferential operators as distributions on $G$. In \cite{DebordSkandalis14} the authors show that one can use instead smooth functions on a richer geometric object that can be called a "blowup groupoid", which comes from the deformation to the normal construction from algebraic geometry, see \cite{DebordSkandalis17}. Note that the use of such deformations is strongly inspired by the pioneer work of Hilsum-Skandalis \cite{HilsumSkandalis87}. Even in the case of the pair groupoid, the result of \cite{DebordSkandalis14} gives a new look on pseudodifferential operators and allows many generalisations. Indeed, it suffices to replace the deformation to the normal cone by another type of deformation to get a new class of operators that are adapted to some geometric constraints. This is, for instance, how van Erp-Yuncken \cite{vanErpYuncken16}, \cite{vanErpYuncken19} defined a pseudodifferential calculus on filtered manifolds. In \cite{AndroulidakisMohsenYuncken} Androulidakis-Mohsen-Yuncken pushed further this idea to prove Helffer-Nourrigat's conjecture, replacing $\DNC(G, \units G)$ by the holonomy groupoid of a singular foliation on $M \times \R_+$, as defined by Debord \cite{Debord01}, \cite{Debord01bis}. This holonomy groupoid is non Hausdorff in general, which generates some analytical technicalities. Later on, Mohsen \cite{Mohsen22}, \cite{Mohsen24} defined a slightly different deformation groupoid and used it to give a more natural proof of Helffer-Nourrigat's conjecture, see \cite{Mohsen26}. The main novelty in the approach of Mohsen was to change the space of units by "blowing-up" some points of $M \times \{0\}$ inside $M \times \R_+$; it allows to get a Hausdorff groupoid without loosing any geometric information. This deformation groupoid is thus of the form 
\begin{equation}
M \times M \times \R_+^* \sqcup \mathcal{G} \times \{0\} \rightrightarrows M \times \R_+^* \sqcup \blup(M) \times \{0\}.
\end{equation}
My research builds on Mohsen's recent work, which has opened up a new and rapidly developping field at the crossroads of noncommutative geometry, analysis of hypoelliptic operators and sub-Riemannian geometry.

\subsection*{Structure of the paper}

In this paper, we generalise Mohsen's construction to any Lie groupoid $G \rightrightarrows \units G$ as follows. Denote by $\bundle A G \rightarrow \units G$ the Lie algebroid of $G$ and choose a finite family of sections $X_1, \dots, X_r \in \ci(\units G, \bundle AG)$. Denote by $\module F^1$ the $\ci(\units G, \R)$-module generated by the $X_i$'s. Then define by induction $\module F^{k+1} \coloneqq \module F^k + [\module F^k, \module F^1]$ and assume that there is an integer $N \geq 1$ such that $\module F^N = \ci(\units G, \bundle A G)$; this is a generalisation of the so-called Hörmander's Lie bracket generating condition. Note that the modules $\module F^k$ may fail to be projective in general. We then build the deformation as follows.
\begin{itemize}
\item In \Cref{section deformation des unites} we define a topological "blowup" of $\units G \times \R_+$ together with a "blowdown map", denoted $\beta: \deformation{\units G}{\module F} \rightarrow \units G \times \R_+$, such that the restriction of $\beta$ to $\beta^{-1}(\units G \times \R_+^*)$ is a homeomorphism. To define $\deformation{\units G}{\module F}$, denote 
\begin{equation*}
\gr(\module F)_p \coloneqq \bigoplus_{k = 1}^N \frac{\module F^k}{\module F^{k-1} + I_p\module F^k}
\end{equation*}
for $p \in \units G$, $I_p$ denoting the functions vanishing at $p$. We define a set $\blowup{\units G}{\module F}_p \subset \Grass(\gr(\module F)_p)$ consisting on well chosen linear subspaces and we set $\deformation{\units G}{\module F} = \units G \times \R_+^* \sqcup_{p \in M} \blowup{\units G}{\module F}_p \times \{0\}$ and we endow this set with a Hausdorff locally compact topology.
\item In \Cref{section deformation algebroide} we build a vector bundle $\deformation{\bundle A G}{\module F} \rightarrow \deformation{\units G}{\module F}$ together with a "blowdown map" $\beta_{\bundle A G} : \deformation{\bundle A G}{\module F} \rightarrow \bundle A G \times \R_+$  which is a vector bundle morphism over $\beta$. This bundle is built so that any section of $\bundle A G \times \R_+$ of the form $t^iX$, with $X \in \module F^i$, lifts through $\beta_{\bundle A G}$. The fibers of $\deformation{\bundle A G}{\module F}$ are
\begin{align*}
\deformation{\bundle A G}{\module F}_{(p, t)} &= \bundle A_p G  &&\text{for} \; (p,t) \in \units G \times \R_+^*; \\
\deformation{\bundle A G}{\module F}_{(p, L, 0)} & = \gr(\module F)_p / L  & &\text{for} \; p \in \units G \; \text{and} \; L \in \blowup{\units G}{\module F}_p.
\end{align*}
Even though $\deformation{\units G}{\module F}$ has not a smooth manifold structure, we define a notion of "smooth" functions and of "smooth" sections of $\deformation{\bundle A G}{\module F}$, a Lie bracket $[\cdot, \cdot]$ and an analogue of an anchor map $\rho$ and show that $(\deformation{\bundle A G}{\module F}, [\cdot, \cdot], \rho)$ behaves as a Lie algebroid.
\item In \Cref{section deformation groupoide} we define our deformation groupoid by noticing that $\gr(\module F)_p$ has a canonical structure of nilpotent Lie algebra and that all $L \in \blowup{\units G}{\module F}_p$ are stable by Lie bracket. Denoting $\Gr(\module F)_p$ the simply connected Lie group integrating $\gr(\module F)_p$, we define a groupoid structure over $\deformation{\units G}{\module F}$ on
\begin{equation*}
\deformation{G}{\module F} = G \times \R_+^* \sqcup \bigsqcup_{p \in M} \bigsqcup_{L \in \blowup{\units G}{\module F}_p} \Gr(\module F)_p / \exp(L)  \times \{0\}.
\end{equation*}
We then endow this groupoid with a locally compact topology and with a set of "smooth" functions. We show that smooth sections of $\deformation{\bundle A G}{\module F}$ act equivariantly on the smooth functions on $\deformation{G}{\module F}$; even though we do not have a Lie groupoid structure, it is thus natural to think about $\deformation{\bundle A G}{\module F}$ as the Lie algebroid of $\deformation{G}{\module F}$.

\item In \Cref{section exemples} we develop variations of Baouendi-Grushin filtration on manifolds with boundary.
\end{itemize}

Note that our construction is a bit more general and includes the case where each generating vector field $X_i$ is weighted by a positive integer; it amounts to modify the definition of the modules $\module F^k$, see \Cref{example generateurs Hormander}. Our general setting is the data of an increasing sequence of locally finitely generated $\ci(\units G, \R)$-modules 
\begin{equation*}
\{0\} = \module F^0 \subseteq \module F^1 \subseteq \cdots \subseteq \module F^N = \ci(\units G, \bundle A G)
\end{equation*}
such that $[\module F^k, \module F^l] \subset \module F^{k+l}$ for all $k,l$; the above construction remains unchanged.

\subsection*{Acknoledgments}
I wish to express my gratitude to my PhD advisors Claire Debord and Omar Mohsen for their support and their numerous remarks during this work.

\subsection*{Notations and conventions}
All along the paper, we denote $\R_+ = [0, +\infty)$, $\R_+^* = (0, + \infty)$ and $\N = \{0, 1, 2, \dots \}$. For $I$ a finite set we will denote by $(e_i)_{i \in I}$ the canonical basis of $\R^I$, namely $e_i = (\delta_{i,j})_{j \in I} \in \R^I$. 

Let $M$ be a manifold, $p \in M$ and $X \in \cci(M, \bundle T M)$. We denote by $\exp(X)p$ the flow of $X$ at time 1 starting at the point $p$.

Let $\bundle E \rightarrow M$ be a smooth real vector bundle and $\module E$ a sub $\ci(M, \R)$-module of $\ci (M, \bundle E)$. We denote by $\module E_{\comp}$ the compactly supported sections of $\module E$ 
\begin{itemize}
\item[•] For $p \in M$ and $X \in \ci(M, \bundle E)$ we say that $X$ \textbf{locally belongs} to $\module E$ around $p$ if $X$ coincides with an element of $\mathcal{E}$ in restriction to a neighbourhood of $p$, or equivalently if there exists a function $f \in \ci(M, \R)$ with $f(p) = 1$ such that $fX \in \module E$.
\item[•] For $U \subseteq M$ an open subset and $(X_i)_{i \in I}$ a family of sections of $\bundle E$, we say that $\module E$ is \textbf{generated by} the $X_i$'s on $U$ if, for any $X \in \module E$, there exist functions $f_i \in \ci (U, \R)$ such that $X_{|U} = \sum_i f_i (X_i)_{|U}$.
\item[•] We say that $\module E$ is \textbf{locally finitely generated} if, for all $p \in M$, there exists an open neighbourhood $U$ of $p$ and a finite family $(X_i)_{i \in I}$ of sections of $\bundle E$ that generates $\module E$ over $U$.
\end{itemize}

Let $G \rightrightarrows \units G$ be a Lie groupoid. The range and source maps will always be denoted respectively by $r$ and $s$ and we denote by $G^{(2)}$ the set of composable pairs in $G$. In this article we will use the following notations.
\begin{itemize}
\item[•] For $U, V$ any pair of subsets of $\units G$ we set $G_U \coloneqq s^{-1}(U)$, $G^V \coloneqq r^{-1}(V)$ and $G_U^V \coloneqq G_U \cap G^V$.
\item[•] For $p,q \in \units G$ and $\gamma \in G_p^q$ we denote by $R_\gamma : G_q \rightarrow G_p$ and $ L_\gamma: G^p \rightarrow G^q$ the diffeomorphisms $R_\gamma(\gamma') = \gamma' \gamma$ and $L_\gamma(\gamma') = \gamma \gamma'$.
\item[•] We denote by $\bundle A G \rightarrow \units G$ the vector bundle given by $(\bundle A G)_p \coloneqq \Ker(ds)_p \subset (\bundle T G)_p$, $p \in \units G$.
\item[•] For $X \in \ci(\units G, \bundle A G)$ we define $X^G \in \ci(G, \Ker(ds))$ the section given by 
\begin{equation}\label{def extension equivariante}
X^G(\gamma) = dR_\gamma \cdot X(r(\gamma)).
\end{equation}
$X^G$ is called the \textbf{right equivariant extension} of $X$. For any pair $X, Y \in \ci(\units G, \bundle A G)$ we define their Lie bracket by
\begin{equation}\label{def crochet de Lie algebroid}
[X,Y] \coloneqq [X^G, Y^G]_{|\units G} \in \ci(\units G, \bundle A G).
\end{equation}
\item[•] Let $X \in \cci(\units G, \bundle AG)$ and $\gamma\in G$; we set 
\begin{equation}\label{def exponentielle}
\exp(X) \gamma\coloneqq \exp(X^G) \gamma.
\end{equation}

Note that since $X^G$ is tangent to the s-fibers one has $s(\exp(X) \gamma)= s(\gamma)$. Moreover denote $\rho\coloneqq dr_{|\bundle A G}: \bundle A G \rightarrow \bundle T M$; one has $r(\exp(X) \gamma) = \exp(\rho(X)) r(\gamma)$. Finally for any pair $(\gamma, \gamma') \in G^{(2)}$ one computes
\begin{equation}\label{eq multiplication groupoide flots}
( \exp(X) \gamma) \gamma' = \exp(X) (\gamma \gamma')
\end{equation}
by right invariance of $X$.
\item The triplet $(\bundle A G, [\cdot, \cdot], \rho)$ is called the \textbf{Lie algebroid} associated to $G$, see \Cref{definition algebroides}.
\end{itemize}

We will keep using \emph{right} equivariant vector fields in the case where $G \rightrightarrows \{1\}$ is a group. We warn the reader used to left invariant vector fields on Lie groups that a few formulas are different from the classical ones, see \eqref{eq BCH}, \eqref{formule ad}.

\section{Deformation of the unit space}\label{section deformation des unites}

\subsection{Localisation of filtrations}
Let $M$ be a smooth manifold and $\bundle E \rightarrow M$ a smooth vector bundle over $M$.

\begin{definition}\label{def filtration}
A (singular) \textbf{filtration} $\module F$ of depth $N \in \N$ of the bundle $\bundle E$ is an increasing family $\{0\} = \module F^0 \subseteq \module F^1 \subseteq \cdots \subseteq \module F^N = \ci(M, \bundle E)$ of sub $\ci(M, \R)$-modules of $\ci (M, \bundle E)$, such that each $\module F^k$ is locally finitely generated. We set $\module F^k = \ci(M, \bundle E)$ for $k \geq N$. 

\end{definition}

\begin{definition}
Let $\module F$ be a filtration of $\bundle E$ of depth $N$. The \textbf{graded localisation} of $\module F$ at $p \in M$ is the vector space
\begin{equation}
\gr(\module F)_p \coloneqq \bigoplus_{k = 1}^N \frac{\module F^k}{\module F^{k-1} + I_p\module F^k}
\end{equation}
where $I_p$ denotes the functions of $\ci(M, \R)$ vanishing at $p$ and $I_p \module F^k$ the linear span of sections of the form $f X$, $f \in I_p$ and $X \in \module F^k$. 

For $X \in \ci(M, \bundle E)$ a section that locally belongs to $\module F^k$ around $p$, we denote by $[X]_{p,k}$ the class of $X'$ in $\frac{\module F^k}{\module F^{k-1} + I_p\module F^k}$, where $X'$ is any section coinciding with $X$ in a neighbourhood of $p$ and such that $X' \in \module F^k$. Moreover we denote by $\alpha$ the linear action of $\R_+^*$ on $\gr(\module F)_p$ defined by
\begin{equation}\label{action de R sur la localisation}
\alpha_\lambda([X]_{p,k}) \coloneqq \lambda^k [X]_{p, k} \quad \forall k \geq 1 \; \forall X \in \module F^k \; \forall \lambda \in \R_+^*.
\end{equation}
\end{definition}

\begin{lemme}[\cite{AndroulidakisSkandalis09}]\label{lemme equivalence generateurs locaux}
Fix $k \geq 1$ and consider a family $(X_j)_{j \in J}$ of elements of $\module F^k$ and a point $p \in M$. The following are equivalent.
\begin{enuminthm}
\item\label{generateur du voisinage} There exists a neighbourhood $U$ of $p$ such that the family $(X_j)_{j \in J}$ generates $\module F^k$ over $U$.
\item\label{generateur du localise} The family $([X_j]_{p,k})_{j \in J}$ generates $\frac{\module F^k}{I_p \module F^k}$.
\end{enuminthm}

\end{lemme}

Since each $\module F^k$ is locally finitely generated, it follows from \Cref{lemme equivalence generateurs locaux} that $\gr(\module F)_p$ is finite dimensional for all $p \in M$. 

\begin{definition}\label{def generateur locaux}
Let $U \subseteq M$ an open set, $ (X_i)_{i \in I}$ a finite family of sections of $\bundle E$ and $k: (X_i)_{i \in I} \rightarrow \N$ a map such that each $X_i$ belongs to $\module F^{k(X_i)}$. We say that $\mathbb{X} = (U, (X_i)_{i \in I}, k)$ is a \textbf{generating family} of the filtration $\module F$ if, for all $l =1, \dots, N$, the family $\{X_i\;|\;k(X_i) \leq l \}$ generates $\module F^l$ over $U$. The map $k$ will be called the \textbf{degree map}. In order to simplify the notations, we usually denote $k(i)$ for $k(X_i)$. 

To any generating family $\mathbb{X}$ is naturally associated the following family of linear maps:

\begin{alignat}{3}
\label{def bekar t non nul} &\text{for} \;(p,t) \in M \times \R_+^* \quad & \natural^{\mathbb X}_{p,t}:  \R^I &\rightarrow \bundle E_p \\
\notag &&e_ i &\mapsto t^{k(i)} X_i(p), \\
\label{def bekar t nul}&\text{for} \; (p,0) \in M \times \{0\}  \quad &\natural^{\mathbb X}_{p,0}:   \R^I &\rightarrow \gr(\module F)_p \\
\notag & &e_i &\mapsto [X_i]_{p,k(i)}
\end{alignat}
where $(e_i)_{i \in I}$ denotes the canonical basis of $\R^I$. Moreover for $t > 0$ and $v \in \R^I$ we denote by $\natural_t^{\mathbb X}(v) \in \ci(M, \bundle E)$ the section $\natural_t^{\mathbb X}(v)(p) \coloneqq \natural_{p,t}^{\mathbb X}(v)$. We will write $\natural_{p,t} = \natural_{p,t}^{\mathbb X}$ and $\natural_t = \natural_t^{\mathbb X}$ when there is no ambiguity on the generating family.
\end{definition}

\begin{remarque}\label{Remarque generateurs locaux semi cont} 
If $(U, (X_i)_{i \in I}, k)$ is a generating family, it follows from \Cref{lemme equivalence generateurs locaux} that the graded localisation $\gr(\module F)_p$ is generated by the family $([X_i]_{p,k(i)})_{i \in I}$ for all $p \in U$. Hence the maps $\natural_{p,0}$ are onto for $p \in U$; even though $\natural_{p,t}$ are defined for any $p \in M$, we will mainly use them when $p \in U$.
\end{remarque}

\begin{exemple}\label{exemple Baouendi classique}
Let $N \geq 2$, $M = \R^2$ and $\bundle E = \bundle T M$ and denote $(\partial_x,\partial_y) \coloneqq (\partial / \partial x, \partial / \partial y)$. The \textbf{Baouendi-Grushin filtration} of depth $N$ is defined by:
\begin{equation}
\module F^k = \left\{f(x,y) \partial_y + g(x,y) y^{N -k}\partial_x  \; | \; f, g \in \ci(M, \R) \right\} ,\quad k = 1, \dots, N.
\end{equation}

Let $p =(x_0, y_0)$. If $y_0 \neq 0$, note that $\partial_x$ locally belongs to $\module F^1$ around $p$; one then computes 
\begin{equation}\label{localisation Baouendi}
\begin{aligned}
\gr(\module F)_p = \begin{cases}
\Span([\partial_y]_{p, 1}, [\partial_x]_{p, 1}) &\text{if} \; y_0 \neq 0\\
\Span([\partial_y]_{p, 1}, [y^{N-1}\partial_x]_{p, 1}, [y^{N-2}\partial_x]_{p, 2}, \dots, [\partial_x]_{p, N}) & \text{if} \; y_0 = 0.
\end{cases}
\end{aligned}
\end{equation}

There are two interesting generating families to be considered in this case:
\begin{enuminthm}
\item\label{base adaptee en y non nul} The family $\mathbb{X} =(\R \times \R^*, (\partial_y, y^{N-1} \partial_x), k )$ with $k(\partial_y) = k(y^{N-1} \partial_x) = 1$. Note that, if $y_0 \neq 0$, $[y^{N-1}\partial_x]_{p,1} = y_0^{N-1} [\partial_x]_{p,1}$.\footnote{This equality does no longer make sense for $y_0 = 0$, since $\partial_x$ does not locally belong to $\module F^1$ around $(x_0, 0)$.}

\item\label{base adaptee en y nul} The family $\mathbb{Y} = (M, (\partial_y, y^{N-1}\partial_x, \dots, y\partial_x, \partial_x), k)$, $k(\partial_y) = 1$ and $k(y^{N-i}\partial_x) = i$ for $i = 1, \dots, N$. Note that, if $y_0 \neq 0$, $[y^{N-i}\partial_x]_{p,i} = 0$ for $i \geq 2$. 
\end{enuminthm}

\end{exemple}

\subsection{Blowup and weighted deformation}

For $E$ a finite dimensional vector space, we denote by $\Grass(E)$ the \textbf{grassmannian manifold}, which is the disjoint union of the manifolds 
\begin{equation}
\Grass_d(E) = \{L \leq E \; |\; \dim(L) = d\}
\end{equation}
for $d = 0, \dots, \dim(E)$. Recall that any onto map between vector spaces $f: E \rightarrow F$ induces the smooth embedding $\Grass_{\dim(F) - d}(F) \hookrightarrow \Grass_{\dim(E) - d}(E)$, $L \mapsto f^{-1}(L)$ for any $d \leq \dim(F)$. The image of this embedding is exactly the submanifold of $\Grass_{\dim(E) - d}(E)$ of all subspaces containing $\Ker(f)$.

The following construction will be useful in many proofs.

\begin{lemme}\label{lemme transitions}
Let $\mathbb{X} = (U, (X_i)_{i \in I}, k)$ and $\tilde{\mathbb{X}} = (\tilde U, (\tilde X_j)_{j \in J}, \tilde k)$ be two generating families. There exists a family of linear maps $T_{p,t} : \R^J \rightarrow \R^I$, depending smoothly on $(p,t) \in U \times \R_+$, satisfying 
\begin{equation}\label{eq prop transition}
\natural_{p,t}^{\tilde{\mathbb{X}}} = \natural_{p,t}^{\mathbb{X}} \circ T_{p,t} \quad \forall (p,t) \in U\times \R_+.
\end{equation}
Such a family will be called a transition family.
\end{lemme}

\begin{proof}
Consider a matrix $(a_{i,j})_{(i,j) \in I \times J}$ with $a_{i,j} \in \ci(U)$ such that $(\tilde X_j)_{|U} = \sum_i a_{i,j} (X_i)_{|U }$ for all $j \in J$ and $a_{i,j} = 0$ if $\tilde k(j) > k(i)$. Then the linear map given by $T_{p,t}(e_j) \coloneqq \sum_i t^{\tilde k(j) - k(i)} a_{i,j}(p) e_i$ satisfies \eqref{eq prop transition}.
\end{proof}

\begin{corollaire}\label{corollaire independance base adaptee} Consider a sequence $(p_n, t_n)_{n \in \N} \in (M \times \R_+^*)^\N$ such that $(p_n, t_n) \rightarrow (p,0)$ (in the topology of $M \times \R_+$) and a subspace $L \in \Grass(\gr(\module F)_p)$. Let $\mathbb{X} = (U, (X_i)_{i \in I}, k)$ and $\tilde{\mathbb{X}} =(\tilde U, (\tilde X_j)_{j \in J}, \tilde k)$ be two generating families such that $p \in U \cap \tilde U$. Then
\begin{equation}\label{equivalence limites generateurs}
\Ker(\natural_{p_n, t_n}^{\mathbb X}) \rightarrow (\natural_{p,0}^{\mathbb X})^{-1}(L) \; \text{in}\; \Grass(\R^I)  \Leftrightarrow  \Ker(\natural_{p_n, t_n}^{\tilde{\mathbb X}}) \rightarrow (\natural_{p,0}^{\tilde{\mathbb X}})^{-1}(L) \; \text{in}\; \Grass(\R^J).
\end{equation}

\end{corollaire}

\begin{definition}\label{def ensemble blup} Let $\mathbb X = (U, (X_i)_{i \in I}, k)$ be a generating family and $p \in U$. Set
\begin{equation}
\begin{split}
\blowup{M}{\module F}_p \coloneqq \left\{L \in \Grass( \gr(\module F)_p)\;|\;\exists (p_n, t_n)_{n \in \N} ; (p_n,t_n) \rightarrow (p,0) \; \text{and} \vphantom{\natural_{p_n,t_n}^X} \right. \\
\left. \Ker(\natural_{p_n,t_n}^{\mathbb X}) \rightarrow (\natural_{p,0}^{\mathbb X})^{-1}(L) \right\}.
\end{split}
\end{equation}

Note that all subspace $L \in \blowup{M}{\module F}_p$ have codimension $\dim(\bundle E_p)$. By \Cref{corollaire independance base adaptee} the set $\blowup{M}{\module F}_p$ is independent of $\mathbb X$; set
\begin{equation}
\blowup{M}{\module F} \coloneqq \bigsqcup_{p \in M} \blowup{M}{\module F}_p.
\end{equation}
\end{definition}

\begin{definition}\label{def topo blup}
Let
\begin{equation}
\deformation{M}{\module F} = M \times \R_+^* \sqcup \blowup{M}{\module F} \times \{0\}
\end{equation} 
and $\beta : \deformation{M}{\module F} \rightarrow M \times \R_+$ the map given by $\beta(p,t) = (p,t)$ for $(p,t) \in M \times \R_+^*$ and $\beta(p,L,0) = (p,0)$ for $(p, L,0) \in \blowup{M}{\module F} \times \{0\}$. To define a topology on $\deformation{M}{\module F}$, let $\mathbb X = (U,(X_i)_{i \in I}, k)$ be a generating family and let:
\begin{equation}\label{def plongement dans la grassmannienne}
\begin{aligned}
\iota^{\mathbb X}: \beta^{-1}(U\times \R_+) & \hookrightarrow U \times \Grass(\R^I) \times \R_+&&\\
(p,t) & \mapsto (p, \Ker(\natural_{p,t}^{\mathbb X}), t) &&\text{if} \; t \neq 0\\
(p,L,0) & \mapsto (p,(\natural_{p,0}^{\mathbb X})^{-1}(L), 0)&& \text{else}.
\end{aligned}
\end{equation}

We will write $\iota = \iota^{\mathbb X}$ when there is no ambiguity on $\mathbb X$. We endow $\deformation{M}{\module F}$ with the coarsest topology such that:
\begin{enuminthm}
\item the subset $M \times \R_+^* \subset \deformation{M}{\module F}$ is open and endowed with its usual topology;
\item the sets $\beta^{-1}(U \times \R_+)$ are open and the maps $\iota^{\mathbb X}$ from \eqref{def plongement dans la grassmannienne} are continuous for every generating family $\mathbb X$.
\end{enuminthm}

We also endow the sets $\blowup{M}{\module F}_p$, $p \in M$ and $\blowup{M}{\module F}$ with the topologies induced by the one of $\deformation{M}{\module F}$.
\end{definition}

\begin{remarque}
Many authors use the notation $[M: N]$ for the blow-up of $M$ with respect to $N$, where $N \subset M$ is a sub-manifold of $M$, see \cite{Melrose09}. \emph{We warn the reader that $\deformation{M}{\module F}$ is not a blow-up of $M \times \R_+$ with respect to a submanifold, in the sense of Melrose}. Nevertheless, we think of $\deformation{M}{\module F}$ as a kind of "weighted blow-up" of $M \times \R_+$ around points of $M\times \{0\}$ where the filtration is "singular"\footnote{By "singular" we mean here that $\module F$ is not locally of the form of the example in \Cref{ex equiregulier}, or equivalently that the dimension of $\gr(\module F)_p$ is not locally constant.}.
\end{remarque}

\begin{proposition}\label{topologie independante de la base adaptee}
The space $\deformation{M}{\module F}$ is locally compact and Hausdorff. Moreover the maps $\iota$ are topological embeddings with closed image in $U \times \Grass(\R^I) \times \R_+$ and the map $\beta$ is a continuous surjection.
\end{proposition}

\begin{proof}
The only thing to prove is that the maps $\iota$ are topological embeddings. Indeed it will directly follow that $\deformation{M}{\module F}$ is locally compact and that $\beta$ is a continuous surjection, since these properties are local. Moreover the Hausdorff property is straightforward and the image of $\iota$ is easily seen to be $\overline{\iota(M \times \R_+^*)}$, hence it is closed in $U \times \Grass(\R^I)\times \R_+$.

Fix $\mathbb X = (U, (X_i)_{i \in I}, k)$ a generating family. We want to show that $\iota^{\mathbb X}(U')$ is open in $\iota^{\mathbb X}(U)$ for any open set $U' \subseteq U$; it is enough to show it for $U' =  (\iota^{\tilde{\mathbb X}})^{-1}(V)$, where $\tilde{\mathbb X} = (\tilde U, (\tilde X_j)_{j \in J}, \tilde k)$ is another generating family with $\tilde U \subseteq U$ and $V \subseteq U \times \Grass(\R^J)\times \R_+ $ is open. We will show it by building a commutative diagram
$$
\begin{tikzcd}
& U \times  \Grass(\R^I) \times \R_+ \ar[rd, "\phi"] & \\
\beta^{-1}(\tilde U \times \R_+) \ar[ru, "\iota^{\mathbb X}"] \ar[rd, " \iota^{ \tilde{\mathbb X}}"] & & U \times  \Grass(\R^I \oplus \R^J ) \times \R_+ \\
& \tilde U \times \Grass(\R^J) \times \R_+ \ar[ru, "\tilde \phi"] & 
\end{tikzcd}
$$
such that $\phi$ and $\tilde \phi$ are topological embeddings. Denoting $\phi_0$ and $\tilde \phi_0$ their restrictions to the images of $\iota^{\mathbb X}$ and $ \iota^{\tilde{\mathbb X}}$, we will thus have $\iota^{\mathbb X}( (\iota^{\tilde{\mathbb X}})^{-1}(V)) = \phi^{-1}_0(\tilde \phi_0(V))$ which is open in $\iota^{\mathbb X}(U)$ as soon as $\phi$ and $\tilde \phi$ are topological embeddings. Consider a transition family $T_{p,t}:\R^J \rightarrow \R^I$ (see \Cref{lemme transitions}) and extend $T_{p,t} : \R^I \oplus \R^J \rightarrow \R^I$ by linearity by setting $T_{p,t} = \id$ on $\R^I$. Reversing $\mathbb X$ and $\tilde{\mathbb X}$ one builds the same way a family $\tilde T_{p,t}: \R^I \oplus \R^J \rightarrow \R^J$. The maps $\phi(p, L, t) = (p,T_{p,t}^{-1}(L), t)$ and $\tilde \phi(p,L,t) = (p,\tilde T_{p,t}^{-1}(L), t)$ then make the diagram commute.
\end{proof}

\begin{exemple}[continuation of \ref{exemple Baouendi classique}]\label{exemple Baouendi blup}
Let $\module F$ be the Baouendi-Grushin filtration (see \Cref{exemple Baouendi classique}); let us compute $\blowup{M}{\module F}_p$ for $p = (x_0, y_0) \in M = \R^2$.

\begin{itemize}
\item If $y_0 \neq 0$, take $\mathbb{X} =(\R \times \R^*, (\partial_y, y^{N-1} \partial_x), k )$ the generating family of \Cref{base adaptee en y non nul}. Then for any $q \in \R \times \R^*$ one computes $\Ker(\natural_{q,t}^{\mathbb X}) = \{0\}$, hence 
\begin{equation}
\blowup{M}{\module F}_p = \{\{0\}\}.
\end{equation}

\item If $y_0 = 0$, take $\mathbb{Y} = (M, (\partial_y, y^{N-1}\partial_x, \dots, y\partial_x, \partial_x), k)$ the generating family of \Cref{base adaptee en y nul} and define
\begin{align}
L_{[a,b]}^p &\coloneqq \Ker\left(\sum_{i=1}^N a^{N-i}b^{i-1} [y^{N-i}\partial_x]_{p,i}^* \right) \cap \Ker \left( [\partial_y]_{p,1}^*\right) \in \Grass(\gr(\module F)_p),\\
\label{def L theta} L_{[a,b]} &\coloneqq (\natural_{p,0}^{\mathbb Y})^{-1}(L_{[a,b]}^p) \\
\notag&= \Ker\left(\sum_{i=1}^N a^{N-i}b^{i-1} e_i^* \right) \cap \Ker \left( e_0^*\right) \in \Grass(\R^{N+1})
\end{align}
with $(e_i^*)_{0 \leq i \leq N}$ (resp $([\partial_y]_{p,1}^*, [y^{N-1}\partial_x]_{p,1}^*, \dots, [\partial_x]_{p,N}^* )$) denoting the dual basis of $(\R^{N+1})^*$ (resp $\gr(\module F)_p^*$) and $[a,b] \in (\R^2 \setminus \{(0,0)\})/\R^* = \mathbb{P}^1(\R)$. By construction, $\Ker(\natural_{(x,y),t}^{\mathbb Y}) = L_{[y, t]}$ for $t>0$

The map $\theta \mapsto L_\theta$ is easily seen to be a topological (and smooth) embedding of $\mathbb{P}^1(\R)$ into $\Grass(\R^{N+1})$; in particular $\{L_\theta \; |\; \theta \in \mathbb{P}^1(\R) \}$ is closed. Therefore, a sequence $\Ker(\natural_{p_n,t_n}^{\mathbb Y})$ converges in $\Grass(\R^{N+1})$ (to a certain $L_\theta$) if and only if $[y_n, t_n]$ converges in $\mathbb{P}^1(\R)$ (to $\theta$). It follows that 
\begin{equation}\label{description blup Baouendi}
\blowup{M}{\module F}_p = \{L_\theta^p \; |\; \theta \in \mathbb{P}^1(\R) \}
\end{equation}
\end{itemize}

Moreover, the map $\iota^{\mathbb{Y}}$ is the composition of the above embedding $\mathbb{P}^1(\R) \hookrightarrow \Grass(\R^{N+1})$ and of the following map:
\begin{equation}\label{eq plongement Baouendi projectif}
\begin{aligned}
i^{\mathbb{Y}} : \deformation{M}{\module F} & \hookrightarrow M \times \mathbb{P}^1(\R) \times \R_+ &&  \\
((x,y),t) & \mapsto ((x,y), [y,t], t) && \text{if} \; t \neq 0 \\
((x,y), \{0\}, 0) & \mapsto ((x,y), [1,0], 0) & &\text{if} \; y \neq 0 \\
((x,0), L_\theta^{(x,0)}, 0) & \mapsto ((x,0), \theta, 0).&&
\end{aligned}
\end{equation}
The map $i^{\mathbb{Y}}$ is then a topological embedding by \Cref{topologie independante de la base adaptee}; it gives a global description of the topology of $\deformation{M}{\module F}$.

Finally, for $p \in \R \times \{0\}$ and $L^p_{[a,b]} \in \blowup{M}{\module F}_p$, one computes from \eqref{action de R sur la localisation} that $\alpha_\lambda(L_{[a,b]}^p) = (L_{[a, \lambda^{-1}b]}^p) \in \blowup{M}{\module F}_p$. We will prove below (see \Cref{action alpha bien def}) that the set $\blowup{M}{\module F}_p \subseteq \Grass(\gr(\module F)_p)$ is actually always closed under the action $\alpha$. In this case, this action has three distincts orbits in $\blowup{M}{\module F}_p$ which are $\{L_{[1,0]}^p\}$, $\{L_{[0,1]}^p\}$ and $\{L_{[a,1]}^p \; |\; a \neq 0\}$.
\end{exemple}

\begin{remarque}
In \cite{Mohsen22} the author asked if the blowdown map $\beta : \deformation{M}{\module F} \rightarrow M \times \R_+$ was open in general. The answer is no and \Cref{exemple Baouendi blup} provides an easy counterexample by considering $\mathcal{U} = M \times \{[a,b]\; |\; a \neq 0 \} \times \R_+ \subset M \times \mathbb{P}^1(\R) \times \R_+$. Indeed $(i^{\mathbb{Y}})^{-1}(\mathcal{U})$ is open, however \eqref{eq plongement Baouendi projectif} shows that $\beta((i^{\mathbb{Y}})^{-1}(\mathcal{U})) = (\R \times \R^* )\times \R_+^* \sqcup M \times \{0\}$, which is not open in $M \times \R_+$.
\end{remarque}

\subsection{Smooth structure}
\emph{\textbf{In general there is no canonical structure of smooth manifold on $\deformation{M}{\module F}$.}} The natural thing to do would be to use the maps $\iota$ to locally embed $\deformation{M}{\module F}$ into a bigger manifold, however, in general, there is no reason for $\iota(\beta^{-1}(U \times \R_+))$ to be a submanifold of $U  \times \Grass(\R^I)\times \R_+$, see \Cref{exemple Baouendi blup}. Nevertheless, we can still define a well-behaved class of "smooth functions" as follows.

\begin{definition}\label{def fonctions lisses unites}
Let $f$ be a continuous (complex valued) function on $\deformation{M}{\module F}$. We say that $f$ is \textbf{smooth} if for any generating family $(U, (X_i)_{i \in I}, k)$:
\begin{equation}\label{existence fonction lisse restriction}
\exists \tilde f \in \ci(U  \times \Grass(\R^I)\times \R_+);\; f_{|\beta^{-1}(U \times \R_+)} = \tilde f_{|\iota(\beta^{-1}(U \times \R_+))} \circ \iota.
\end{equation}
We denote by $\ci(\deformation{M}{\module F})$ the set of (complex valued) smooth functions. 

More generally, for $N$ a smooth manifold, a continuous function $f$ on $\deformation{M}{\module F} \times N$ is said to be smooth if, replacing $U  \times \Grass(\R^I)\times \R_+$ by $U  \times \Grass(\R^I)\times \R_+ \times N$, \eqref{existence fonction lisse restriction} holds for all generating families. The set of such functions will be denoted $\ci(\deformation{M}{\module F} \times N)$. 

Finally, a continuous map from $\deformation{M}{\module F}$ to $N$ (or from $\deformation{M}{\module F}$ to itself) is called smooth if it pulls back smooth functions to smooth functions.
\end{definition}

Using transition families (see \Cref{lemme transitions}) one shows easily the following.

\begin{lemme}
Let $f$ be a continuous (complex valued) function on $\deformation{M}{\module F}$. If, for every point $p \in M$, there is a generating family $(U, (X_i)_{i \in I}, k)$ such that $U$ contains $p$ and \eqref{existence fonction lisse restriction} holds, then $f \in \ci(\deformation{M}{\module F})$. 

This property extends in the obvious way to smooth functions on $\deformation{M}{\module F} \times N$, for any smooth manifold $N$.
\end{lemme}

\begin{definition}\label{def structure lisse fibre}
A \textbf{smooth vector bundle} over $\deformation{M}{\module F}$ is a topological vector bundle $\bundle F \rightarrow \deformation{M}{\module F}$, endowed with an atlas such that all transition maps are given by (matrix-valued) smooth functions; such an atlas is called smooth. Given a smooth structure on $\bundle F$, we denote by $\ci(\deformation{M}{\module F}, \bundle F)$ the set of continuous sections of $\bundle F$ that restrict to (vector valued) smooth functions in every chart of the smooth atlas.

We also extend this definition in the obvious way to vector bundles over $\deformation{M}{\module F} \times N$, for any smooth manifold $N$. Morphisms of smooth vector bundles are also defined in the obvious way in smooth atlas.
\end{definition}

\subsection{Debord-Skandalis action}
\begin{definition}\label{def action debord skandalis}
Recall that, for all $p \in M$, there is a canonical action of $\R_+^*$ on $\gr(\module F)_p$ denoted by $\alpha$ and defined by \eqref{action de R sur la localisation}. We will still denote by $\alpha$ the action $\R_+^* \curvearrowright \deformation{M}{\module F}$ defined by 
\begin{equation}\label{action Debord Skandalis}
\begin{aligned}
\alpha_\lambda(p,t) &= (p, \lambda^{-1} t) && \text{for} \; (p, t) \in M \times \R_+^* \\
\alpha_\lambda(p, L, 0) &= (p, \alpha_\lambda(L), 0) && \text{for} \; (p, L) \in \blowup{M}{\module F}
\end{aligned}
\end{equation}
for all $\lambda \in \R_+^*$. The action $\alpha$ is called the \textbf{Debord-Skandalis action}.
\end{definition}

\begin{lemme}\label{action alpha bien def}
The action $\alpha$ is well defined, ie $\alpha_\lambda(L)$ belongs to $\blowup{M}{\module F}_p$ for all $(p,L) \in \blowup{M}{\module F}$ and all $\lambda > 0$, and is continuous. Furthermore, the action $\alpha$ is smooth in the sense that, for all $f \in \ci(\deformation{M}{\module F})$, the map $(a, \lambda) \mapsto f(\alpha_\lambda(a))$ belongs to $\ci(\deformation{M}{\module F} \times \R_+^*)$.
\end{lemme}

\begin{proof}
Let $(U, (X_i)_{i \in I}, k)$ be a generating family and define $\tilde \alpha $ the linear action of $\R_+^*$ on $\R^I$ given by $\tilde \alpha_{\lambda}(e_i) \coloneqq \lambda^{k(i)} e_i$. We still denote $\tilde \alpha$ the (continuous) action of $\R_+^*$ on $M \times \Grass(\R^I) \times \R_+$ given by $\tilde \alpha_\lambda(p,L,t) = (p, \tilde \alpha_\lambda(L), \lambda^{-1}t)$. Notice that $\iota(\alpha_\lambda(p,t)) = \tilde \alpha_\lambda (\iota(p,t))$ for all $(p,t) \in U \times \R_+^*$. Since $\iota(U \times \R_+^*)$ is a dense subset of the image of $\iota$, the image of $\iota$ is stable under the action $\tilde \alpha$. The continuity of $\alpha$ then follows from the fact that $\iota$ is a local topological embedding by \Cref{topologie independante de la base adaptee}, and the smoothness of $\alpha$ comes from the smoothness of $\tilde \alpha$.
\end{proof}

\section{Deformation of the algebroid}\label{section deformation algebroide}
\subsection{Deformation of the bundle}

\begin{proposition}\label{def deformation du fibré}
Set
\begin{align}
\Lie{osc}_{\module F}(\bundle E) &= \bigsqcup_{(p,L) \in \blowup{M}{\module F}} \gr(\module F)_p / L,\\
\deformation{\bundle E}{\module F} &= \bundle E \times \R_+^* \sqcup \Lie{osc}_{\module F}(\bundle E)\times \{0\}
\end{align}
We call $\Lie{osc}_{\module F}(\bundle E)$ the \textbf{osculating bundle} and $\deformation{\bundle E}{\module F}$ the \textbf{deformation bundle} associated to $\module F$. The elements of $\Lie{osc}_{\module F}(\bundle E)$ are thus pairs of the form $(p,u \mod L)$ with $(p,L) \in \blowup{M}{\module F}$ and $u \in \gr(\module F)_p$.

Let $\pi: \deformation{\bundle E}{\module F} \rightarrow \deformation{M}{\module F}$ be defined by $\pi(p,X, t) = (p,t)$ if $t \neq 0$ and $\pi(p,u \mod L, 0) = (p,L, 0)$. There is a canonical structure of smooth vector bundle on $(\deformation{\bundle E}{\module F}, \pi)$, in the sense of \Cref{def structure lisse fibre}, such that the following holds.
\begin{enuminthm}
\item\label{structure fibre loin de zero} The identity map between $\deformation{\bundle E}{\module F}_{|M \times \R_+^*}$ and $\bundle E \times \R_+^*$ is a diffeomorphism.
\item For any $k \in \N$ and any $X \in \module F^k$, the section of $\deformation{\bundle E}{\module F}$ given by 
\begin{equation}\label{eq theta champs de vecteurs}
\begin{aligned}
\theta_k(X)(p,t) & = t^k X(p) && \text{if} \; t \neq 0\\
\theta_k(X)(p,L,0) & = (p, [X]_{p, k} \mod L)&&  \text{else}
\end{aligned}
\end{equation}
is smooth.
\end{enuminthm}
The sections $\theta_k(X)$ with $X \in \module F^k$ thus locally generate $\ci(\deformation{M}{\module F}, \deformation{\bundle E}{\module F})$ over $\ci(\deformation{M}{\module F}, \R)$, since they generate $\deformation{\bundle E}{\module F	}$ fiberwise. 

Finally, denote
\begin{equation}
\begin{tikzcd}
\deformation{\bundle E}{\module F} \ar[r, "\beta_{\bundle E}"] \ar[d] & \bundle E \times \R_+ \ar[d]\\
\deformation{M}{\module F} \ar[r, "\beta"] & M \times \R_+
\end{tikzcd}
\end{equation}
the map given, fiberwise, by $(\beta_{\bundle E})_{|\bundle E\times \R_+^*} = \id$ and $(\beta_{\bundle E})_{|\Lie{osc}_{\module F}(\bundle E)} \equiv 0$. Then $\beta_{\bundle E}$ is a morphism of smooth vector bundles and, for any $X \in \module F^k$, the section $t^k X \in \ci(M \times \R_+, \bundle E \times \R_+)$ lifts to $\theta_k(X)$ through $\beta_{\bundle E}$.
\end{proposition}

\begin{proof}
The condition \ref{structure fibre loin de zero} ensures the uniqueness of the bundle structure, by density of $M \times \R_+^*$ in $\deformation{M}{\module F}$. It thus suffices to build the trivialisations of the bundle in the neighbourhood of points $(p,L_p, 0)\in \deformation{M}{\module F}$, and to check that these trivialisations are compatible with one another and compatible with $\deformation{\bundle E}{\module F}_{|M \times \R_+^*} \simeq \bundle E \times \R_+^*$. 

Let $(U, (X_i)_{i \in I}, k)$ be a generating family with $p \in U$ and denote $\CoTaut \rightarrow \Grass(\R^I)$ be the "co-tautological" bundle over $\Grass(\R^I)$, ie the vector bundle whose fiber over $L$ is $\R^I/L$. We build a bijection $\Phi$ between $\deformation{\bundle E}{\module F}_{|\beta^{-1}(U \times \R_+)}$ and the pullback bundle $\iota^*(U  \times \CoTaut \times \R_+)$ the following way:
\begin{equation}\label{trivialisation deformation fibre}
\begin{aligned}
\Phi: (U  \times \CoTaut \times \R_+)_{|\iota(\beta^{-1}(U \times \R_+))} & \rightarrow \deformation{\bundle E}{\module F}_{|\beta^{-1}(U \times \R_+)}&&\\
(p, \Ker(\natural_{p,t}), v \mod \Ker(\natural_{p,t}) , t) & \mapsto (p, \natural_{p,t}(v),t) && \text{if} \; t \neq 0 \\
 (p,L, v \mod L, 0)&  \mapsto (p, \natural_{p,0}(v) \mod \natural_{p,0}(L), 0) &&\text{else}.
\end{aligned}
\end{equation}
The trivialisation $\Phi$ is clearly compatible with the structure of $\deformation{\bundle E}{\module F}_{|M \times \R_+^*}$ since the map $U \times \R^I \times \R_+^* \rightarrow \bundle E \times \R_+^*$, $(p,v,t) \mapsto (\natural_{p,t}(v), t)$ is smooth and lifts $\Phi$. Moreover, if we are given another generating family $(\tilde U, (\tilde X_j)_{j \in J}, \tilde k)$, giving rise to another trivialisation $\tilde \Phi$, the compatibility between $\Phi$ and $\tilde \Phi$ follows using a transition family as given by \Cref{lemme transitions}.

Finally consider $X \in \module F^l$, written locally $X = \sum_i a_i X_i$ where the $a_i$'s are smooth functions on $U$ such that $a_i = 0$ if $k(i) > l$, and set $Y$ the smooth section of $U \times \CoTaut \times \R_+$ given by $Y(p,L,t) = \sum_i t^{l - k(i)} a_i(p) e_i \mod L$. Then by construction $\Phi^{-1}(\theta_l(X))$ is the restriction of $Y$ to $\iota(\beta^{-1}(U \times \R_+))$, hence $\theta_l(X)$ is smooth. The asserted proporties of $\beta_{\bundle E}$ are then straightforward to check locally, in the smooth atlas given by \eqref{trivialisation deformation fibre}.
\end{proof}

\subsection{Filtration of Lie algebroids}
Recall the following definition of Lie algebroid.

\begin{definition}\label{definition algebroides}
A \textbf{Lie algebroid} over $M$ is a triplet $(\bundle A, [\cdot, \cdot], \rho)$, where $\bundle A \rightarrow M$ is a smooth vector bundle, $[\cdot, \cdot] : \ci(M, \bundle A) \times \ci(M, \bundle A) \rightarrow \ci(M, \bundle A)$ a $\R$-bilinear map and $\rho: \bundle A \rightarrow \bundle T M$ a vector bundle morphism, satisfying the following conditions:
\begin{enuminthm}
\item $[\cdot, \cdot]$ is a Lie bracket ie it is antisymmetric and satisfies Jacobi identity;
\item for any pair of sections $X,Y \in \ci(M, \bundle A)$ the relation 
\begin{equation}\label{rho est un morphisme de lie}
\rho([X,Y]) = [\rho(X), \rho(Y)]
\end{equation}
holds, where $\rho$ still denotes the induced map $\rho: \ci(M, \bundle A) \rightarrow \ci(M, \bundle TM)$;
\item\label{Leibniz algebroides} for any pair of sections $X,Y \in \ci(M, \bundle A)$ and any function $f \in \ci(M, \R)$ one has
\begin{equation}\label{equation Leibniz}
[X, fY] = f[X, Y] + (\rho(X) \cdot f)Y.
\end{equation}
\end{enuminthm}

The map $\rho$ is called the \textbf{anchor} of the algebroid and \eqref{equation Leibniz} will be refered to as Leibniz rule. 
\end{definition}

Let $(\bundle A \rightarrow M, [ \cdot, \cdot ], \rho)$ be a Lie algebroid.

\begin{definition}\label{def filtration algebroide}
By a (singular) filtration of $(\bundle A, [ \cdot, \cdot ], \rho)$, we mean a filtration $\module F = (\module F^k)_{k \in \N}$ of the bundle $\bundle A \rightarrow M$, in the sense of \Cref{def filtration}, satisfying the relation
\begin{equation}\label{condition crochets def filtration}
 [\module F^k, \module F^l] \subseteq \module F^{k+l} \quad \forall k, l \in \N.
\end{equation} 
We sometimes call $\module F$ a \textbf{Lie filtration} of $\bundle A$ when the Lie algebroid structure is implicit.
\end{definition}

\begin{exemple}\label{example generateurs Hormander}
Given $X_1, \dots, X_r \in \ci(M, \bundle A)$ set $\module F^1 = \langle X_1, \dots, X_r \rangle $ and define, by induction, $\module F^{k+1} \coloneqq \langle \module F^k, [\module F^k, \module F^1] \rangle$, where $\langle \cdot \rangle$ denotes "the $\ci(M, \R)$ module generated by $\cdot$". These modules always satisfy \eqref{condition crochets def filtration}, thus they define a Lie filtration if and only there exists a rank $N \geq 1$ such that $\module F^N = \ci(M, \bundle A)$. This condition on the $X_i$'s is a generalisation, to Lie algebroids, of the so-called Hörmander's Lie bracket generating condition.

More generally, let $X_1, \dots, X_r \in \ci(M, \bundle A)$ and $w_1, \dots, w_r \in \N \setminus \{0\}$. For $d \geq 2$ and $I = (i_1, \dots, i_d) \subset \{1, \dots, r\}^d$, denote $X_I \coloneqq [X_{i_1}, [X_{i_2},[ \dots[X_{i_{d-1}}, X_{i_d} ]]\dots ]$ and $w_I = w_{i_1} + \dots + w_{i_d}$. Then set $\module F^k \coloneqq \langle \{ X_I \; |\; I \subset \{1, \dots, r\}^d; w_I \leq k , d \geq 1\} \rangle$. The previous case corresponds to $w_1 = \dots = w_r = 1$. If one assumes that $\module F^N = \ci(M, \bundle A)$ at a certain rank $N$, then $\module F$ defines a Lie filtration; any finitely generated filtration is of this form.
\end{exemple}

Let $\module F$ be a filtration on $(\bundle A, [\cdot, \cdot], \rho)$, $p \in M$ and $k,l \in \N$. By Leibniz rule \eqref{equation Leibniz} the Lie bracket $[\cdot, \cdot]: \module F^k \times \module F^l \rightarrow \module F^{k+l}$ induces a bilinear map 
\begin{equation}
[\cdot, \cdot]: \frac{\module F^k}{\module F^{k-1} + I_p\module F^k} \times \frac{\module F^l}{\module F^{l-1} + I_p\module F^l} \rightarrow \frac{\module F^{k+l}}{\module F^{k+l-1} + I_p\module F^{k+l}}.
\end{equation}

In other words, for $X \in \module F^k$ and $Y \in \module F^l$, the bracket given by $[[X]_{p,k}, [Y]_{p,l}] \coloneqq [[X,Y]]_{p, k+l}$ is well defined. Extending $[\cdot, \cdot]$ to $\gr(\module F)_p$ by bilinearity turns the localisation into a graded nilpotent Lie algebra of degree $N$. Moreover the compatibility between the filtration and the bracket implies the following.

\begin{lemme}\label{Blup stable par crochet}
For all $(p, L) \in \blowup{M}{\module F}$ the space $L$ is a sub Lie algebra of $\gr(\module F)_p$.
\end{lemme}

The proof of \Cref{Blup stable par crochet} relies on the following construction.

\begin{lemme}\label{lemme relevement crochet}
Let $(U, (X_i)_{i \in I}, k)$ be a generating family. There exists a family of bilinear maps $S_{p,t} : \R^I \times \R^I \rightarrow \R^I$, depending smoothly on $(p,t) \in U \times \R_+$, satisfying the relations
\begin{align}
\label{S releve le crochet}\natural_{p,t}(S_{p,t}(u,v)) &= [\natural_t(u), \natural_t(v)](p) && \text{if} \; t \neq 0, \\
\label{S vaut presque le crochet en q zero}\natural_{p,0}(S_{p,0}(u,v)) &= [\natural_{p,0}(u),\natural_{p,0}(v)] && \text{else}
\end{align}
for all $u, v \in \R^I$ and all $(p,t) \in U \times \R_+$. Moreover one can assume that, for any $p \in U$ and $v \in \R^I$, the linear map $u \mapsto S_{p,0}(u,v)$ is nilpotent of degree $N$.
\end{lemme}

\begin{proof}
Consider a family of smooth functions $(a_{i,j,u})_{(i,j,u) \in I^3}$ with $a_{i,j,u} \in \ci(U)$ such that $[X_i, X_j]_{|U} = \sum_u a_{i,j,u} X_u$ for all $(i,j) \in I^2$ and $a_{i,j,u} = 0$ if $k(u) > k(i) + k(j)$. For $(p,t) \in U \times \R_+$, the bilinear map defined by 
\begin{equation}
S_{p,t}(e_i, e_j) = \sum_{u \in I} t^{k(i) + k(j) - k(u)} a_{i,j,u}(p) e_u 
\end{equation}
satisfies the required conditions. Moreover one has $S_{p, 0}(e_i, v) \in \Span\{ e_u\; |\; k(u) > k(i) \}$ for any $p \in U$, $v \in \R^I$ and $i \in I$. Decomposing $\R^I$ as the direct sum of $\Span\{ e_i \; |\; k(i) = l\}$, $l = 1, \dots, N$, the map $u \mapsto S_{p,0}(u, v)$ is thus strictly lower triangular and hence $N$-nilpotent.
\end{proof}

\begin{proof}[Proof of \Cref{Blup stable par crochet}]
Let $(U, (X_i)_{i \in I}, k)$ be a generating family and $(S_{q,t})_{(q,t) \in U \times \R_+}$ a family of bilinear maps as given by \Cref{lemme relevement crochet}. Consider a sequence $(p_n, t_n)$ such that $(p_n, t_n) \rightarrow (p,0)$ and $\Ker(\natural_{p_n, t_n}) \rightarrow \natural_{p,0}^{-1}(L)$. Take $u,v \in L$, lift them through $\natural_{p,0}$ to $\tilde u, \tilde v \in \R^I$ and consider a sequence $\tilde u_n, \tilde v_n \in \Ker(\natural_{p_n, t_n})$ such that $\tilde u_n \rightarrow \tilde u$ and $\tilde v_n \rightarrow \tilde v$. By continuity of $S_{q,t}$ with respect to $(q,t)$, one has $ S_{p_n,t_n}(\tilde u_n, \tilde v_n) \rightarrow S_{p,0}(\tilde u,\tilde v)$. Moreover, since the Lie bracket of two sections vanishing at a point still vanishes at this point, it follows from \eqref{S releve le crochet} that $\natural_{p_n,t_n}(S_{p_n,t_n}(\tilde u_n, \tilde v_n)) = 0$ for all $n \in \N$, hence $S_{p,0}(\tilde u, \tilde v) \in \lim\Ker(\natural_{p_n, t_n})= \natural_{p,0}^{-1}(L)$ so $[u,v] = \natural_{p,0}(S_{p,0}(\tilde u, \tilde v)) \in L$.
\end{proof}

\begin{definition}
For $p \in M$, let $\Gr(\module F)_p$ be the simply connected Lie group integrating $\gr(\module F)_p$. It means that the exponential map $\exp: \gr(\module F)_p \rightarrow \Gr(\module F)_p$ is a diffeomorphism and that the group structure is given by
\begin{equation}
\forall u,v \in \gr(\module F)_p\quad e^u\cdot e^v = e^{\BCH(u,v)}
\end{equation} where 
\begin{equation}\label{eq BCH}
\BCH(u,v) = u + v - \frac{1}{2}[u,v] + \frac{1}{12} [u,[u,v]] + \dots
\end{equation}
is the Baker-Campbell-Hausdorff formula. Note that the use of right invariant vector fields changes some signs compared to the classical BCH formula. 
\end{definition}

By differentiating the action of $\Gr( \module F)_p$ on itself by conjugation, one gets a canonical action of $\Gr( \module F)_p$ on $\gr(\module F)_p$ denoted
\begin{equation}
gvg^{-1} \coloneqq \left. \frac{d}{dt}\right|_{t=0} ge^{tv}g^{-1}.
\end{equation}
for $g \in \Gr(\module F)_p$ and $v \in \gr(\module F)_p$

\Cref{Blup stable par crochet} shows that any $(p,L) \in \blowup{M}{\module F}$ defines a subgroup $\exp(L) \subset \Gr(\module F)_p$. However these subgroups are not normal in $\Gr(\module F)_p$ in general, which means that we may have $gLg^{-1} \neq L$ for some $g \in \Gr(\module F)_p$, see \Cref{ancre paires}. Nevertheless, the set of subspaces $\blowup{M}{\module F}_p$ is closed under conjugation by any element $g \in \Gr(\module F)_p$, as we will show in \Cref{blup stable par conjugaison}.

\subsection{Deformation of Lie algebroids}\label{section deformation algebroide}

Since a filtration of $(\bundle A, [\cdot,\cdot], \rho)$ is, in particular, a filtration of the bundle $\bundle A \rightarrow M$ in the sense of \Cref{def filtration}, one can extend $\bundle A$ to a "smooth" bundle $\deformation{\bundle A}{\module F} \rightarrow \deformation{M}{\module F}$ by \Cref{def deformation du fibré}. Since the topological space $\deformation{M}{\module F}$ is not a smooth manifold, it does not make sense to talk about a Lie algreboid structure on $\deformation{\bundle A}{\module F}$. Indeed the anchor map should be a bundle morphism with values in the "tangent space" of $\deformation{M}{\module F}$, and this "tangent space" is not well-behaved in general (see \cite[Section 1.1]{Mohsen26}). However, the bundle $\deformation{\bundle A}{\module F}$ still satisfies the following.

\begin{proposition}\label{extension de l'exponentielle}
For  $Z$ a section of $\deformation{\bundle A}{\module F}$, denote $Z_0 = \beta_{\bundle A *}(Z_{|M \times \R_+^*}) \in \ci(M \times \R_+^*, \bundle A \times \R_+^*)$\footnote{The map $\beta$ being an isomorphism when restricted to $M \times \R_+^*$, the pushforward of sections by $\beta_{\bundle A}$ makes sense.}.
\begin{enuminthm}
\item\label{existence flot exponentielle} For $Y$ a section of $\deformation{\bundle A}{\module F}$  with compact support, there is a unique diffeomorphism of $\deformation{M}{\module F}$ (in the sense of \Cref{def fonctions lisses unites}), denoted $ \exp_\rho(Y)$, such that the diagram
\begin{equation}
\begin{tikzcd}[column sep={150,between origins}]
M \times \R_+^* \ar[r, " \exp_\rho(Y)_{|M \times \R_+^*}"] \ar[d, "\beta"]& M \times \R_+^* \ar[d, "\beta"]\\
M \times \R_+^* \ar[r, "\exp(\rho(Y_0))"] & M \times \R_+^*
\end{tikzcd}
\end{equation}
commutes.
\item\label{derivation ancre} The map $\rho(Y): \ci(\deformation{M}{\module F}) \rightarrow \ci(\deformation{M}{\module F})$ given by 
\begin{equation}
\rho(Y)\cdot f\coloneqq \left.\frac{d}{ds}\right|_{s = 0} f \circ \exp_\rho(sY)
\end{equation}
is a well defined derivation. When $Y$ has not compact support anymore, the flow $\exp_\rho(sY)$ is still well defined for small $s$ (depending on the point of $\deformation{M}{\module F}$) hence $\rho(Y)$ is defined for any smooth section $Y$.
\item\label{extension crochet} There is a canonical Lie bracket on $\ci(\deformation{M}{\module F}, \deformation{\bundle A}{\module F})$ such that $[Y,Y']_0 = [Y_0, Y_0']$\footnote{The Lie bracket on $\ci(M \times \R_+^*, \bundle A \times \R_+^*)$ is inherited from the one of $\ci(M,\bundle A)$} for any pair of sections $(Y,Y')$. Moreover
\begin{equation}
\rho([Y,Y']) = \rho(Y) \rho(Y') - \rho(Y') \rho(Y).
\end{equation}
and 
\begin{equation}
[X, fY] = f[X, Y] + (\rho(X) \cdot f)Y
\end{equation}
for any $f \in \ci(\deformation{M}{\module F}, \R)$.
\item\label{blup stable par conjugaison} If $Y$ is of the form $Y = \sum_i (a_i \circ \beta) \theta_{k_i}(Y_i)$, with $a_i \in \cci(M \times \R_+, \R)$ and $(Y_i, k_i)_{i \in I}$ any finite family satisfying $Y_i \in \module F^{k_i}$, then
\begin{equation}\label{flot exponentiel conjugaison}
 \exp_\rho(Y)(p,L,0) = (p, gLg^{-1}) \quad \forall (p,L,0) \in \blowup{M}{\module F} 
\end{equation}
where $g = e^{\sum a_i(p,0) [Y_i]_{k_i}} \in \Gr(\module F)_p$. In particular it means that for any $(p,L) \in \blowup{M}{\module F}$ and any $g  \in \Gr(\module F)_p$ the space $gLg^{-1}$ still belongs to $\blowup{M}{\module F}_p$. 
\end{enuminthm}
\end{proposition}

\begin{remarque}
The condition of \Cref{blup stable par conjugaison} means that $Y$ is a lift, through $\beta_{\bundle A}$, of the vector field $\sum_i a_i t^{k(i)} Y_i \in \cci(M \times \R_+, \bundle T M)$.
\end{remarque}

The proof of \Cref{extension de l'exponentielle} relies on general facts about the manifold structure of the Grassmannian space. Consider a vector space $E$, fix an integer $d \leq \dim(E)$ and denote by $\Grass(E)$ the Grassmannian manifold of subspaces of dimension $d$. Then for $L \in \Grass(E)$ there is a canonical identification 
\begin{equation}\label{tangent grassmannienne}
\bundle T_L \Grass(E) \simeq \mathcal{L}(L, E/L)
\end{equation}
as follows: for any smooth family $(L_t)_{t \in \R}$ of subspaces with $L_0 = L$ set
\begin{equation}
\left. \frac{d L_t}{dt} \right|_{t = 0} : l \mapsto l'(0) \mod L
\end{equation}
where $l(t)$ is any smooth path with $l(t) \in L_t$ for all $t$ such that $l(0) = l$. In particular any linear map $\varphi \in \mathcal{L}(E,E)$ defines a vector field $X_\varphi$ on $\Grass(E)$ by
\begin{equation}\label{champ endomorphisme grassmannienne}
X_\varphi(L) : l \mapsto \varphi(l) \mod L.
\end{equation}
One easily computes $\exp(X_\varphi) L = e^{\varphi}(L)$. 

Moreover for $H$ a Lie group, $v \in \Lie h$ and $g \in H$, denote $\ad(v) \coloneqq[\cdot, v] \in \mathcal{L}(\Lie h, \Lie h)$\footnote{Be careful that the Lie bracket on $\Lie g$ is defined using right invariant vector fields, which is why the map $\ad$ differs by a sign from the classical one.} and 
\begin{equation}
\label{formule ad} 
gwg^{-1} \coloneqq \left. \frac{d}{ds} \right|_{s = 0} g e^{sw} g^{-1}.
\end{equation}
These two maps are related by the classical formula $e^v w e^{-v} = e^{\ad(v)}(w)$  for all $v,w \in \Lie h$, where the exponential of the LHS is the Lie group one and the exponential of the RHS is the exponential in $\mathcal{L}(\Lie h, \Lie h)$. In particular
\begin{equation}\label{flot grassmannienne}
\exp(X_{ad(v)}) L = e^v L e^{-v} \quad \forall v \in \Lie g
\end{equation}
in $\Grass(\Lie h)$.

\begin{proof}[Proof of \Cref{extension de l'exponentielle}]

Uniqueness follows from the density of $M \times \R_+^*$ in $\deformation{M}{\module F}$. For the existence, let $(U, (X_i)_{i \in I}, k)$ be a generating family. Set $\mathcal{U} = \{(p,t) \in M \times \R_+^* \;|\; \exp(\rho(Y_0)) (p,t) \in U \times \R_+^* \} \sqcup U \times \{0\}$. Since $Y_0$ continuously extends by $0$ on $M \times \{0\}$, the set $\mathcal{U}$ is open; we want to define $\exp_\rho(Y) : \beta^{-1}(\mathcal{U}) \rightarrow \beta^{-1}(U \times \R_+)$. 

For $(p,t) \in U \times \R_+^*$ denote
\begin{equation}
\begin{aligned}
d\iota_{(p,t)}\cdot \rho(Y_0)&= (\rho(Y_0)(p,t), \phi, 0) \\
&\in \bundle T_{(p,\Ker(\natural_{p,t}), t)}(U \times \Grass(\R^I) \times \R_+^*).
\end{aligned}
\end{equation}
We claim that $\phi$ identifies through \eqref{tangent grassmannienne} with the unique map from $\Ker(\natural_{p,t})$ to $\R^I / \Ker(\natural_{p,t})$ satisfying
\begin{equation}\label{equation differentielle iota}
\natural_{p,t}(\phi(v)) =[\natural_t(v), Y_0(\cdot, t)](p).
\end{equation}
Before proving \eqref{equation differentielle iota} let us show how it implies \Cref{extension de l'exponentielle}. 

First write $Y$ (restricted to $\beta^{-1}(U \times \R_+)$) as $Y = \sum_i b_i \theta_{k(i)}(X_i)$ with $b_i \in \ci(\beta^{-1}(U \times \R_+), \R)$, consider smooth functions $\tilde b_i \in \ci(U \times \Grass(\R^I) \times \R_+, \R)$ such that $\tilde b_i \circ \iota = b_i$ and set $w(p,L,t) = \sum_i \tilde b_i(p,L,t)e_i \in \R^I$ for $(p,L,t) \in U \times \Grass(\R^I) \times \R_+$. Then consider a smooth family of bilinear maps $(S_{p,t})_{(p,t) \in U \times \R_+}$ as given by \Cref{lemme relevement crochet} and define $Z$ the following smooth vector field on $U \times \Grass(\R^I) \times \R_+$:

\begin{equation}
\begin{aligned}
Z(p, L, t) &= (\rho(\beta_{\bundle A}(Y(p,L,t))), (v \mapsto S_{p,t}(v,w(p,L,t)) \mod L), 0) \\
&\in \bundle T_{(p,L, t)}(U \times \Grass(\R^I) \times \R_+^*).
\end{aligned}
\end{equation}
It follows from \eqref{S releve le crochet} and \eqref{equation differentielle iota} that $Z_{|\iota(U \times \R_+^*)} = d\iota \cdot \rho(Y_0)$. Indeed, for $(p,t) \in U \times \R_+^*$ and $v \in \Ker(\natural_{p,t})$:
\begin{equation}
\natural_{p,t}(S_{p,t}(v,w(p,L,t))) - [\natural_t(v), Y_0(\cdot, t)](p)= [\natural_t(v), \natural_t(w(p,L,t)) - Y_0(\cdot, t)](p) = 0
\end{equation}
since the Lie bracket of two vector fields vanishing at $p$ still vanishes at $p$, and $\natural_t(w(p,L,t))(p) = Y_0(p,t)$ by construction. Hence one can define $\exp_\rho(Y)$ on $\beta^{-1}(\mathcal{U})$ by the relation $\iota \circ \exp_\rho(Y) = \exp(Z)$; \ref{existence flot exponentielle} and \ref{derivation ancre} then follow directly since $Z$ is smooth. 

To prove \ref{extension crochet}, write $Y' = \sum_i c_i \theta_{k(i)}(X_i)$ and set
\begin{align*}
[Y, Y'] &= \sum_{i,j} [b_i \theta_{k(i)}(X_i), c_j \theta_{k(j)}(X_j)]\\
&=\sum_{i,j} b_i (\rho(\theta_{k(i)}(X_i)) \cdot c_j )\theta_{k(j)}(X_j) - c_j (\rho(\theta_{k(j)}(X_j)) \cdot b_i )\theta_{k(i)}(X_i)\\
& + b_i c_j \theta_{k(i) + k(j)}([X_i, X_j])
\end{align*}
which is well defined by \ref{derivation ancre} and satisfies the asserted relations by construction. 

Now assume that $Y$ is of the form of \ref{blup stable par conjugaison}. Up to complete $(U, (Y_i), (k_i))$ in a generating family, it amounts to assume that the coefficients $b_i$ are of the form $b_i = a_i \circ \beta$. One can thus take $\tilde b_i(p,L,t) = a_i(p,t)$ and hence $w(p,L,t) = w(p,t)$: at $t = 0$ the vector field $Z$ is thus a family, parametrized by $p$, of vector fields on $\{p\} \times \Grass(\R^I) \times \{0\}$ of the form of \eqref{champ endomorphisme grassmannienne} with $\varphi(v) = S_{p,0}(v, w(p, 0))$. Set $u_p = \natural_{p,0}(w(p,0)) \in \gr(\module F)_p$: by \eqref{S vaut presque le crochet en q zero}, $\natural_{p,0} \circ \varphi = \ad(u_p) \circ \natural_{p,0}$ hence $\natural_{p,0} \circ e^\varphi = e^{\ad(u_p)} \circ \natural_{p,0}$. \ref{blup stable par conjugaison} then follows from \eqref{flot grassmannienne}.

It remains to prove \eqref{equation differentielle iota}. Fix $(p,t) \in U \times \R_+^*$, choose a subset $J \subset I$ such that $(X_j(p))_{j \in J}$ forms a linear basis of $\bundle A_p$ and consider an element $v \in \Ker(\natural_{p,t})$. In a neighbourhood of $p$ there are uniquely defined smooth functions $(f_j)_{j \in J}$ on $M$ such that $\natural_t(v) = \sum_{j \in J} f_j t^{k_j}X_j$ with all $f_j$'s vanishing at $p$, since $\natural_{p,t}(v) = 0$; using Leibniz rule \eqref{equation Leibniz} one thus computes $[\natural_t(v),Y_0(\cdot, t)](p) = - \sum_{j \in J} (\rho(Y_0(p,t)) \cdot f_j) t^{k_j}X_j(p)$. On the other side set $v_s = v - \sum_{j \in J} f_j(p_s) e_j$ where $p_s = \exp(s\rho(Y_0(\cdot, t))) p$: then $v_0 = v$ and for all $s$ small enough the vector $v_s$ belongs to $\Ker(\natural_{p_s, t})$. Hence:
\begin{align*}
\natural_{p,t}(\phi(v)) & = \left. \frac{d}{ds}\right|_{s = 0} \natural_{p,t}(v_s)\\
& = \left. \frac{d}{ds}\right|_{s = 0} - \sum_{j \in J} f_j(p_s) t^{k_j} X_j(p)\\
& = - \sum_{j \in J} \rho(Y_0(p,t)) \cdot f_j t^{k_j}X_j(p)\\
& = [\natural_t(v), Y_0](p,t)
\end{align*}
which completes the proof.
\end{proof}

\begin{exemple}[continuation of \ref{exemple Baouendi classique} and \ref{exemple Baouendi blup}]\label{ancre paires}
Consider $\module F$ the Baouendi-Grushin filtration of depth $N$ on $M = \R^2$ and $(\bundle T M, [\cdot, \cdot ], \rho)$ the Lie algebroid with $[ \cdot, \cdot]$ the usual Lie bracket of vector fields and $\rho = \id: \bundle T M \rightarrow \bundle TM$. It is straightforward to check that $\module F$ is a Lie filtration for this structure. For $p \in M$, the Lie algebra structure on $\gr(\module F)_p$ is the following:

\begin{itemize}
\item if $p \in \R \times \R^*$ then $ \gr(\module F)_p = \Span([\partial_y]_{p,1}, [\partial_x]_{p,1})$ thus the Lie bracket is identically zero;
\item if $p \in \R \times \{0\}$ then $\gr(\module F)_p = \Span([\partial_y]_{p,1}, [y^{N-1}\partial_x]_{p,1}, \dots, [\partial_x]_{p,N})$ and the Lie bracket is
\begin{equation}
\begin{aligned}
[[\partial_y]_{p,1}, [y^{N-i}\partial_x]_{p,i}] &= (N-i)[y^{N-i-1}\partial_x]_{p,i+1} && \forall i = 1, \dots, N-1\\
[[\partial_y]_{p,1}, [\partial_x]_{p,N}] &= [[y^{N-i}\partial_x]_{p,i}, [y^{N-j}\partial_x]_{p,j}] = 0 && \forall i,j = 1, \dots, N.
\end{aligned}
\end{equation}
\end{itemize}

Now let us compute, using \Cref{blup stable par conjugaison}, the action $\Gr(\module F)_p \curvearrowright \blowup{M}{\module F}_p$ at points $p \in \R \times \{0\}$ (the action is trivial if $p \in \R \times \R^*$ since $\Gr(\module F)_p$ is abelian). 

Recall (see \Cref{exemple Baouendi blup}) that $\blowup{M}{\module F}_p = \{L_\theta^p \; |\; \theta \in \mathbb{P}^1(\R) \}$; for $v = \sum_i b_i [y^{N-i}\partial_x]_{p,i} + c[\partial_y]_{p,1} \in \gr(\module F)_p$ and $\theta \in \mathbb{P}^1(\R)$, we want to compute $e^vL_\theta^p e^{-v}$. Let $(p_n, t_n) \in M \times \R_+^*$ be a sequence converging to a certain $ (p, L_\theta^p, 0) \in \blowup{M}{\module F} \times \{0\}$, $p \in \R \times \{0\}$, for the topology of $\deformation{M}{\module F}$; denoting $p_n = (x_n, y_n)$, it is equivalent (by \eqref{eq plongement Baouendi projectif}) to say that $(p_n,t_n, [y_n, t_n]) \rightarrow (p,0, \theta)$ for the topology of $M \times \R_+ \times \mathbb{P}^1(\R)$.

Then let $f$ be a compactly supported function on $M$ with value $1$ in a neighbourhood of $p$ and consider the section of $\deformation{\bundle T M}{\module F}$ $X = \sum_i b_i \theta_i(f y^{N-i}\partial_x) + c\theta_1(f \partial_y)$ ($\theta_i$ being defined by \eqref{eq theta champs de vecteurs}). Finally let $(p_n', t_n) = \exp_\rho(X)(p_n, t_n)$ with $p_n' = (x_n', y_n')$. 

One directly computes $y_n' = y_n + t_nc$ (for $t_n$ small enough). Denoting $\theta = [a,b]$ it follows that $(p_n', t_n)$ converges to $(p, L_{[a + bc,b]}^p, 0)$ in $\deformation{M}{\module F}$. Hence \Cref{blup stable par conjugaison} implies that $e^vL_{[a,b]}^p e^{-v} = L_{[a+bc, b]}^p$; one can easily check that it is consistent with the BCH formula. Note in particular that the subgroup $\exp(L_\theta^p)$ is normal in $\Gr(\module F)_p$ only for $\theta = [1,0]$.
\end{exemple}

\section{Deformation of the groupoid}\label{section deformation groupoide}

\subsection{The deformation groupoid}

Recall that for any Lie groupoid $G \rightrightarrows \units G$, the Lie bracket defined by \eqref{def crochet de Lie algebroid} and the anchor $\rho \coloneqq dr_{|\bundle A G}$ turn $\bundle A G \rightarrow \units G$ into a Lie algebroid (see for example \cite{MoerdijkMrcun03}). Let $G$ be a Lie groupoid endowed with a filtration $\module F  = (\module F^k)_{k \geq 0}$ of the algebroid $(\bundle A G, [ \cdot, \cdot], \rho)$, in the sense of \Cref{def filtration algebroide}.

\begin{definition}
The \textbf{osculating groupoid} associated to the filtration $\module F$ is the groupoid over $\blowup{\units G}{\module F}$ denoted by $\Osc_{\module F}(G)$ and defined as a set by 
\begin{equation}
\Osc_{\module F}(G) \coloneqq \bigsqcup_{(L,p) \in \blowup{\units G}{\module F}} \Gr(\module F)_p/e^L \rightrightarrows \blowup{\units G}{\module F}
\end{equation}
where $e^L = \{e^v \; | \; v \in L\}$ is a (non necessary normal) subgroup of $\Gr(\module F)_p$. The elements of $\Osc_{\module F}(G)$ are thus pairs of the form $(p, g \mod e^L)$ with $(p,L) \in \blowup{\units G}{\module F} $ and $g \in \Gr(\module F)_p$. The groupoid structure is then given by 
\begin{itemize}
\item[•] $s(p, g \mod e^L) = (p, L)$ and $r(p, g \mod e^L) = (p, g L g^{-1})$
\item[•] $(p, g' \mod ge^Lg^{-1})\cdot(p,g \mod e^L) = (p,g'g \mod e^L)$
\item[•] $(p, g \mod e^L)^{-1} = (p, g^{-1} \mod ge^Lg^{-1})$ 
\end{itemize}
which is well defined by \Cref{blup stable par conjugaison}. Note that, at this point, $\Osc_{\module F}(G)$ is only a set-theoretic groupoid and does not have any topology or smooth structure.
\end{definition}

\begin{definition}\label{def groupoid de deformation}
The \textbf{deformation groupoid} associated to the filtration $\module F$ is the groupoid over $\deformation{\units G}{\module F}$ denoted by $\deformation{G}{\module F}$ and defined as a set by 
\begin{equation}
\deformation{G}{\module F} \coloneqq G \times \R_+^* \sqcup \mathcal \Osc_{\module F}(G) \times \{0\} \rightrightarrows \deformation{\units G}{\module F}
\end{equation}
with groupoid structure inherited by the ones of $G\times \R_+^* \rightrightarrows \units G \times \R_+^*$ and $\Osc_{\module F}(G) \rightrightarrows \blowup{\units G}{\module F}$.

Let $\mathbb X = (U, (X_i)_{i \in I}, k)$ be a generating family such that all $X_i$ are compactly supported. Set 
\begin{equation}\label{quotients exponentielle}
\begin{aligned}
\Exp^{\mathbb X} : \beta^{-1}(U \times \R_+)\times \R^I & \rightarrow \deformation{G}{\module F}\\
(p, t, v) & \mapsto (\exp(\natural_t^{\mathbb X}(v)) p, t) \quad \text{if} \; t \neq 0 \\
(p, L, 0, v) & \mapsto (p, e^{\natural_{p,0}^{\mathbb X}(v)} \mod L, 0)
\end{aligned}
\end{equation}
where $\exp$ is understood in the sense of \eqref{def exponentielle}. We will write $\Exp = \Exp^{\mathbb X}$ when there is no ambiguity on $\mathbb X$. 

We endow $\deformation{G}{\module F}$ with the finest topology such that:
\begin{enuminthm}
\item\label{topologie t non nul} the inclusion $G \times \R_+^* \hookrightarrow \deformation{G}{\module F}$ is continuous,
\item\label{topologie t nul} the maps $\Exp^{\mathbb X}$ from \eqref{quotients exponentielle} are continuous for every $\mathbb X$.
\end{enuminthm}
In other words, we use the final topology associated to the above maps. It is equivalent to ask that a function $f$ on $\deformation{G}{\module F}$ is continuous if and only if:
\begin{enuminthm}
\item $f_{|G \times \R_+^*}$ is continuous (for the product topology of $G \times \R_+^*$);
\item for any generating family, the map $f \circ \Exp$ is continuous.
\end{enuminthm}

We also endow $\Osc_{\module F}(G)$ with the topology inherited as a sub-groupoid of $\deformation{G}{\module F}$.
\end{definition}

\begin{remarque}\label{remarque restriction topologie}
Since the restrictions $\Exp_{|U \times \R_+^* \times \R^I} \rightarrow G \times \R_+^*$ are continuous for the product topology of $G \times \R_+^*$, one easily checks that, by \Cref{topologie t non nul}, the topology of $G \times \R_+^*$ as a subspace of $\deformation{G}{\module F}$ is just the usual one. Moreover, it means that one gets the same topology on $\deformation{G}{\module F}$ up to reduce the domains of the maps $\Exp$ to any open subset $\mathcal{U} \subset \beta^{-1}(U \times \R_+) \times \R^I$ containing $\beta^{-1}(U \times \{0\}) \times \R^I$.
\end{remarque}

\begin{theoreme}\label{theoreme topologie groupoide}
The space $\deformation{G}{\module F} $ is locally compact and Hausdorff, the maps of source, range, inversion, inclusion of units and multiplication are continuous and the source and the range maps are open. 
\end{theoreme}

\begin{remarque}
The groupoid  $\deformation{G}{\module F} \rightrightarrows \deformation{\units G}{\module F}$ is thus a locally compact groupoid, as defined in \cite{Renault80}\footnote{Renault does not include the openness of the source and range maps in his definition; however he showed that it is a necessary condition for the existence of a Haar system.}.
\end{remarque}

The key idea in the proof of \Cref{theoreme topologie groupoide} is to lift certain "well-behaved" equivariant vector fields $Y$ on $\deformation{G}{\module F}$ through $\Exp$, as vector fields $\tilde Y$ on $\beta^{-1}(U \times \R_+) \times \R^I$, and then to follow the flow of $\tilde Y$. Before proving \Cref{theoreme topologie groupoide}, we shall start by explaining this lifting procedure.

\subsection{Lie algebroid of the deformation}

We think of $(\deformation{\bundle A G}{\module F}, [\cdot, \cdot], \rho)$ as the Lie algebroid of the groupoid $\deformation{G}{\module F}$ ($[\cdot, \cdot]$ and $\rho$ being defined by \Cref{extension de l'exponentielle}). Even though $\rho$ is not a bundle morphism, hence $(\deformation{\bundle A G}{\module F}, [\cdot, \cdot], \rho)$ is not a Lie algebroid in the classical sense and $\deformation{G}{\module F}$ is not a Lie groupoid, we still have the following.

\begin{proposition}\label{prop flot s fibres}
Let $Y \in \ci(\deformation{\units G}{\module F}, \deformation{\bundle A G}{\module F})$ a section of the form $Y = \sum_j (a_j \circ \beta) \theta_{k_j}(Y_j)$, with $a_j \in \ci(M \times \R_+, \R)$ and $(Y_j, k_j)_{j \in J}$ a finite family satisfying $Y_j \in \module F^{k_j}_{\comp}$.
\begin{enuminthm}
\item\label{existence extension} The homeomorphism $\exp(Y_{|\units G \times \R_+^*}): G \times \R_+^* \rightarrow G \times \R_+^*$ uniquely extends to a homeomorphism $\exp(Y) : \deformation{G}{\module F} \rightarrow \deformation{G}{\module F}$.
\item\label{description extension en zero} For $(p, g \mod L, 0) \in \Osc_{\module F}(G) \times \{0\}$ one has 
\begin{equation}
\exp(Y)(p, g \mod L, 0) = (p, e^{\tilde Y(p)}g \mod L, 0)
\end{equation}
where $\tilde Y(p) = \sum_j a_j(p) [Y_j]_{p, k_j}$. Note that $\tilde Y(p) \in \gr(\module F)_p$ is a lift of $Y(p, L, 0) \in \gr(\module F)_p / L$.
\item\label{source but exponentielle} For any $\gamma \in \deformation{G}{\module F}$ one has $s(\exp(Y)\gamma) = s(\gamma)$ and $r(\exp(Y) \gamma) = \exp_\rho(Y)r(\gamma)$.
\item\label{associativite exponentielle} For any composable pair $(\gamma, \gamma') \in \deformation{G}{\module F}^{(2)}$, one has 
\begin{equation}
(\exp(Y)\gamma)\gamma' = \exp(Y)(\gamma\gamma').
\end{equation}
\end{enuminthm}
\end{proposition}

\Cref{prop flot s fibres} is a corollary of the following lemma.

\begin{lemme}\label{lemme lift champs de vecteurs Exp}
Let $(U, (X_i)_{i \in I}, k)$ be a generating family such that all $X_i$'s are compactly supported. There exists $\mathcal{U} \subseteq U \times \R_+ \times \R^I$ an open neighbourhood of $U \times \{0\} \times \R^I$ such that the following holds. 

Let $(c_i)_{i \in I}$ be a family of functions $c_i \in \ci(\units G \times \R_+, \R)$ and $Z = \sum_i c_i t^{k(i)} X_i \in \ci(\units G \times \R_+, \bundle A G \times \R_+)$. There exists $\tilde Z$ a vector field on $\mathcal{U}$, such that:
\begin{enuminthm}
\item for all $(p,t,v) \in \mathcal{U}$, $\tilde Z(p,t, v)$ is tangent to $\{(p,t)\} \times \R^I$;
\item\label{Z tilde t non nul} for all $(p,t, v) \in \mathcal{U} \cap (U \times \R_+^* \times \R^I)$, one has 
\begin{equation}
(d\Exp \cdot \tilde Z)(p,t,v) = Z^{G \times \R_+^*}(\Exp(p,t,v)),
\end{equation}
where $Z^{G \times \R_+^*}$ denotes the equivariant extension of $Z$ to the groupoid $G \times \R_+^* \rightrightarrows \units G \times \R_+^*$, see \eqref{def extension equivariante};
\item\label{Z tilde t nul} for all $(p,0,v) \in \mathcal{U} \cap (U \times \{0\} \times \R^I)$, one has 
\begin{equation}
\natural_{p,0}(\tilde Z (p,0,v))= \left. \frac{d}{d\tau}\right|_{\tau = 0} \BCH(\tau w(p), \natural_{p,0}(v))
\end{equation}
where $w(p) = \sum_i c_i(p) [X_i]_{p, k(i)} \in \gr(\module F)_p$.
\end{enuminthm}

Note that, in \ref{Z tilde t non nul}, the map $\Exp: \mathcal{U} \cap (U \times \R_+^* \times \R^I) \rightarrow \deformation{G}{\module F}$ is well defined even though $\mathcal{U}$ is not a subset of $\beta^{-1}(U \times \R_+) \times \R^I$, since $U \times \R_+^*$ canonically embeds into $\beta^{-1}(U \times \R_+)$.
\end{lemme}

\begin{proof}
Let $S_{p,t} : \R^I \times \R^I \rightarrow \R^I$, $(p,t) \in U \times \R_+$ be a family as given by \Cref{lemme relevement crochet}. For $(p,t, v) \in U \times \R_+ \times \R^I$ let
\begin{equation}
p(t,v) = \exp(\rho( \natural_t(v)))p 
\end{equation}
and $\mathcal{V} = \{(p,t,v) \in U \times \R_+ \times \R^I \; |\; p(t,\tau v) \in U \; \forall \tau \in [-1,1] \}$. For $(p,t,v) \in \mathcal{V}$, consider the path $(f_{p,t,v}(\tau))_{\tau \in [-1, 1]}$, $f_{p,t,v}(\tau) \in \mathcal{L}(\R^I)$ that solves the linear ODE 
\begin{align}
f_{p,t,v}(0) &= \id \\
\label{EDO f}\frac{\partial f_{p,t,v}}{\partial \tau} \cdot w &= f_{p,t,v}(\tau) \cdot S_{p(t,-\tau v),t}(v,w) \quad \forall w \in \R^I.
\end{align}
Now set 
\begin{equation}
h_{p, t, v} = \int_0^1 f_{p,t,v}(\tau) d\tau \in \mathcal{L}(\R^I)
\end{equation}
Since $p(0, u) = p$ for all $(p,u) \in U \times \R^I$, one has
\begin{equation}
f_{p,0,v}(\tau) = e^{\tau S_{p,0}(v, \cdot)}.
\end{equation}
Recall that one can choose the family $S_{p,t}$ such that $S_{p,0}(v, \cdot)$ is nilpotent for all $(p,v) \in U \times \R^I$. Since $x \mapsto (e^x - 1)/x$ has value $1$ at $x= 0$, the endomorphisms $h_{p, 0,v}$ are invertible for all $(p,0,v) \in \mathcal{V} \cap( U \times \{0\} \times \R^I)$; set 
\begin{equation}
\mathcal{U} = \left\{ (p,t,v) \in \mathcal{V} \; |\; h_{p(1,v),t,v} \in \Gl(\R^I) \right\} .
\end{equation}
For $(p,t, v) \in \mathcal{U}$, let $w(p,t) = \sum_i c_i(p,t) e_i \in \R^I$ and set
\begin{equation}
\tilde Z(p,t,v) = \left(h_{p(1,v), t,v} \right)^{-1} \cdot w(p(1,v),t) \in \R^I \subset \bundle T_{p,t,v} (\mathcal{U}).
\end{equation}
Let $w_0 \in \R^I$ and $q = p(1,v)$. We claim that 
\begin{equation}\label{eq derivee chemin h}
\left. \frac{d}{ds}\right|_{s=0} \exp(\natural_t( v + sw_0)) \exp(-\natural_t(v)) q = \natural_{q,t}(h_{q, v,t} \cdot w_0).
\end{equation} 
This claim is essentially the second equation of (3.26) in \cite{Mohsen26}, we refer to Mohsen for the detailed proof.

Set $w_0 = \tilde Z (p,t,v)$ and $\gamma = \Exp(p,t,v) = \exp(\natural_t(v))p$, notice that $r(\gamma) = q$ and denote $R_\gamma : G_q \rightarrow G_p$ the right multiplication by $\gamma$. We thus get from \eqref{eq derivee chemin h}:
\begin{align*}
&(d\Exp \cdot \tilde Z)(p,t,v)\\
=& \left. \frac{d}{ds}\right|_{s=0} \exp(\natural_t(v + s\tilde Z (p,t,v))) p\\
=& \left. \frac{d}{ds}\right|_{s=0} R_\gamma \left(  \exp(\natural_t(v + s\tilde Z (p,t,v))) \exp(-\natural_t(v)) \right)\\
=& \left(d_q R_\gamma \right) \natural_{q,t}(h_{q, v,t} \cdot \tilde Z (p,t,v))\\
=&\left(d_q R_\gamma \right) \natural_{q,t}(w(q, t)) \\
=& Z^{G \times \R_+^*}(\gamma)
\end{align*}
hence \ref{Z tilde t non nul} holds.

Finally \ref{Z tilde t nul} is a consequence of the identity $\natural_{p,0} \circ S_{p,0}(v, \cdot) = \ad_{\natural_{p,0}(v)} \circ \natural_{p,0}$, see \Cref{S vaut presque le crochet en q zero}, and of the general formula for the derivative of the exponential on a Lie group, see \cite[Theorem 5 Section
1.2]{Rossmann06}.
\end{proof}

\begin{proof}[Proof of \Cref{prop flot s fibres}]

Let $(U, (X_i)_{i \in I}, k)$ be a generating family such that all $X_i$'s are compactly supported, $\mathcal{U}$ the open set given by \Cref{lemme lift champs de vecteurs Exp} and $(c_i)_{i \in I}$ a family of functions $c_i \in \ci(U \times \R_+, \R)$ such that $Y_{|\beta^{-1}(U \times \R_+)} = \sum_i (c_i \circ \beta) \theta_{k(i)}(X_i)$ (such a family always exists by decomposing the $Y_j$'s in the family $(X_i)_{i \in I}$). Set $\mathcal{U}' = \{(a, v) \in \beta^{-1}(U \times \R_+) \times \R^I \; |\; (\beta(a), v) \in \mathcal{U} \}$ and lift the vector field $\tilde Z$ on $\mathcal{U}$, given by \Cref{lemme lift champs de vecteurs Exp} applied to the family of functions $(c_i)_{i \in I}$,   as a vector field $\hat Z$ on $\mathcal{U}'$. The diagram 
\begin{equation}\label{eq relevement champs de vecteurs}
\begin{tikzcd}
\mathcal{U}'  \ar[r, "\exp(\hat Z)"] \ar[d, "\Exp"'] & \beta^{-1}(U \times \R_+) \times \R^I \ar[d, "\Exp"] \\
\deformation{G}{\module F} \ar[r, "\exp(Y)"] & \deformation{G}{\module F}
\end{tikzcd}
\end{equation}
then commutes by construction, with $\exp(Y)$ the map described in \ref{existence extension} and \ref{description extension en zero}; it directly implies the continuity of $\exp(Y)$ in the neighbourhood of $\Osc_{\module F}(G) \times \{0\}$. The points \ref{source but exponentielle} and \ref{associativite exponentielle} are then straightforward by definition of $\exp(Y)$.
\end{proof}

\begin{corollaire}\label{approximation du produit}
Let $(U, (X_i)_{i \in I}, k)$ be a generating family such that all the $X_i$'s are compactly supported and $\mathcal{U}\subseteq U \times \R_+ \times \R^I$ the open neighbourhood of $U \times \{0\} \times \R^I$ given by \Cref{lemme lift champs de vecteurs Exp}. For $(p,t) \in U \times \R_+$, set $V_{p,t} = \{v \in \R^I \; |\; (p,t,v) \in \mathcal{U} \}$. 

There exists a family of smooth maps $\phi_{p,t}: \R^I \times V_{p,t} \rightarrow \R^I$, depending smoothly on $(p,t) \in U \times \R_+$\footnote{In the sense that $\mathcal{U} \times \R^I \mapsto \R^I$, $((p,t,v),u) \mapsto \phi_{p,t}(u,v)$ is smooth.}, satisfying the relations
\begin{align}
\label{eq phi composition t non nul}\exp(\natural_t(\phi_{p,t}(u,v)))p &= \exp(\natural_t(u))\exp(\natural_t(v))p && \text{if} \; t\neq 0\\
\label{eq phi composition t nul}e^{\natural_{p,0}(\phi_{p,0}(u,v))} &= e^{\natural_{p,0}(u)} e^{\natural_{p,0}(v)} && \text{else}
\end{align}
for all $(p,t, v) \in \mathcal{U}$ and $u \in \R^I $. Moreover one can assume that, for any fixed $p \in U$ and any $u \in \R^I$, the map $\R^I \rightarrow \R^I$, $v \mapsto \phi_{p,0}(u, v)$ is an isomorphism.

\end{corollaire}

\begin{proof}

For $u = \sum_i c_i e_i \in \R^I$ ($(c_i)_{i \in I}$ being constants), set $Z_u = \sum_i c_i t^{k(i)} X_i$ and define $\phi_{p,t}(u,v)$ by the relation
\begin{equation}
(p, t, \phi_{p,t}(u,v)) = \exp(\tilde Z_u)(p,t,v)
\end{equation}
where $\tilde Z_u$ is the lift given by \Cref{lemme lift champs de vecteurs Exp}.
\end{proof}

\subsection{Proof of local compactness}

\begin{lemme}\label{lemme ouvert independant base}
A set $\mathcal{V} \subseteq \deformation{G}{\module F}$ is open if and only if:
\begin{enuminthm}
\item $\mathcal{V} \cap (G \times \R_+^*)$ is open and
\item for all $a \in \mathcal{V} \cap(\Osc_{\module F}(G) \times \{0\})$, there exists a generating family $(U, (X_i)_{i \in I}, k)$ such that $s(a) \in \beta^{-1}(U \times \{0\})$ and $(\Exp^X)^{-1}(\mathcal{V})$ is open.
\end{enuminthm}
\end{lemme}

\begin{proof}

Let $\mathbb X = (U, (X_i)_{i \in I}, k)$, $\tilde{\mathbb X} = (\tilde U, (\tilde X_j)_{j \in J}, \tilde k)$ be two generating families such that all the $X_i$'s and $\tilde X_j$'s are compactly supported.

We will build $\mathcal{W} \subseteq (U \cap \tilde U) \times \R_+ \times \R^I$ an open neighbourhood of $(U \cap \tilde U) \times \{0\} \times \R^I$ and maps $\Psi$, $\tilde \Psi$ such that the diagram 
\begin{equation}
\begin{tikzcd}
\mathcal{W} \ar[rr, "\Psi"] & & (U \cap \tilde U) \times \R_+ \times \R^J\\
\mathcal{W}' \ar[rr, "\tilde \Psi"] \ar[rd, "\Exp^{\mathbb X}"'] \ar[u, "\beta"]&& \beta^{-1}((U \cap \tilde U) \times \R_+) \times \R^J \ar[ld, "\Exp^{\tilde{\mathbb X} }"] \ar[u, "\beta"] \\
& \deformation{G}{\module F}&
\end{tikzcd}
\end{equation}
commutes, where $\mathcal{W}' = \{(a,v) \in \beta^{-1}((U \cap \tilde U) \times \R_+) \times \R^I \; |\; (\beta(a), v) \in \mathcal{W} \}$. The continuity of $\tilde \Psi$ will end the proof.

First apply \Cref{lemme transitions} to $\mathbb X$ and $\tilde{\mathbb X}$ to get the family of maps $T_{p,t} : \R^I\rightarrow \R^J$, and linearly extend these maps as $T_{p,t} : \R^I \oplus \R^J \rightarrow \R^J$ by the identity on $\R^J$. Then consider the generating family $\mathbb Y = (U \cap \tilde U, (X_i)_{i \in I} \sqcup (\tilde X_j)_{i \in J}, k')$ with $k'_{|I} = k$ and $k'_{|J} = \tilde k$. Apply \Cref{approximation du produit} to $\mathbb Y$ to get an open set $\mathcal{U} \subseteq (U \cap \tilde U) \times \R_+ \times ( \R^I \oplus \R^J)$ and maps $\phi_{p,t} : (\R^I \oplus \R^J) \times V_{p,t} \rightarrow \R^I$, with $V_{p,t} \subseteq \R^I \oplus \R^J$. Set $\mathcal{W}_0 =\{(p,t,v) \in (U \cap \tilde U) \times \R_+ \times \R^I \; |\; (p,t, v \oplus 0) \in \mathcal{U} \}$. For $(p,t,v) \in \mathcal{W}_0$, $w \in \R^J$ and $a \in \beta^{-1}(\{(p,t)\})$, one gets from the properties of $\phi_{p,t}$ and $T_{p,t}$ that 
\begin{equation}
T_{p,t}(\phi_{p,t}(-w, v)) = 0 \Rightarrow \Exp^{\tilde{\mathbb X} }(a,w) = \Exp^{\mathbb X}(a, v).
\end{equation} 

Thus, writing $\Psi(p,t,v) = (p,t,\psi_{p,t}(v))$, we want $\psi_{p,t}(v) \in \R^J$ to be a solution of the equation $T_{p,t}(\phi_{p,t}(- \psi_{p,t}(v), v)) = 0$. Consider the smooth map 
\begin{equation}
\begin{aligned}
\mathcal{W}_0 \times \R^J & \overset{F}{\rightarrow} \mathcal{W}_0 \times \R^J\\
F(p,t,v, w) &=(p,t,v, T_{p,t}(\phi_{p,t}(-w, v))).
\end{aligned}
\end{equation}
By the same argument as in the proof of \cite[Lemma 3.3.2]{Mohsen26}, one can find $\mathcal{D} \subseteq \mathcal{W}_0 \times \R^J$ an open neighbourhood of $(U \cap \tilde U) \times \{0\} \times \R^I \times \R^J$ such that $F: \mathcal{D} \rightarrow F(\mathcal{D})$ is a diffeomorphism and such that $F(\mathcal{D})$ contains $(U \cap \tilde U) \times \{0\} \times \{0\}$. One can thus set $\mathcal{W} = \{(p,t,v) \in \mathcal{W}_0\; |\; (p,t,v,0) \in F(\mathcal{D})\}$ and define $\Psi$ by the relation 
\begin{equation}
(p,t, \psi_{p,t}(v), v) = F^{-1}(p,t,v, 0).
\end{equation}

Note that the key argument in Mohsen's proof is the $\R_+^*$ stability of the domain $\mathcal{U}$ and the $\R_+^*$ equivariance of the family of maps $\phi_{p,t}$, for the suitable action. Our construction is essentially the same as Mohsen's and we can also assume this equivariance. More precisely it is a consequence of the fact that the families of maps $T_{p,t}$ and $S_{p,t}$ can be chosen $\R_+^*$-equivariant for the suitable $\R_+^*$ actions, see the proof of \Cref{lemme relevement crochet}.
\end{proof}

\begin{corollaire}
A function $f$ on $\deformation{G}{\module F}$ is continuous if and only if:
\begin{enuminthm}
\item $f_{|G \times \R_+^*}$ is continuous;
\item\label{pullback exponentielle lisse} for every point $p \in M$, there is a generating family $\mathbb X = (U, (X_i)_{i \in I}, k)$ such that $U$ contains $p$ and the map $f \circ \Exp^{\mathbb X}$ is continuous.
\end{enuminthm}
\end{corollaire}

We finally need the so-called \emph{period bounding lemma}; it will be the key argument to show the Hausdorff property of $\deformation{G}{\module F}$. See \cite[Lemma 3.3.2]{Mohsen26} for a proof.

\begin{lemme}\label{periodic bounding}(period bounding lemma)
Let $(U, X_i)_{i \in I}, k)$ be a generating family with all $X_i's$ being compactly supported. There exists $W \subseteq \R^I$ an open neighbourhood of $0$ such that, for all $v \in W$ and all $p \in U$:
\begin{equation}
\exp(\natural_1(v))p = p \Leftrightarrow \natural_{p,1}(v) = 0.
\end{equation}
\end{lemme}

\begin{proof}[Proof of \Cref{theoreme topologie groupoide}] Let $\mathbb X = (U, (X_i)_{i \in I}, k)$ be a generating family with all $X_i's$ compactly supported, $\mathcal{U}$ the open set given by \Cref{lemme lift champs de vecteurs Exp} and $\phi_{p,t}: \R^I \times V_{p,t} \rightarrow \R^I$ a family as given by \Cref{approximation du produit}. Moreover, for $t \in \R_+$, denote $\tilde \alpha_t$ the endomorphism of $\R^I$ given by $\tilde \alpha_t(e_i) = t^{k(i)} e_i$.

\begin{itemize}
\item \emph{Openness of $\Exp$ around $t=0$: } We claim that we can build an open set $\mathcal{V} \subset \beta^{-1}(U \times \R_+) \times \R^I$, containing $\beta^{-1}(U \times \{0\}) \times \R^I$, such that the restriction $\Exp_{|\mathcal{V}}: \mathcal{V} \rightarrow \deformation{G}{\module F}$ is open. Together with the Hausdorff property, proven below, it will imply local compactness of $\deformation{G}{\module F}$ by \Cref{remarque restriction topologie}, see \cite[Chap I, \S 10, Prop 10]{BourbakiTG}. 

To build $\mathcal{V}$, start by choosing an open set $\mathcal{V}_1 \subset U \times \R^I$ containing $U \times \{0\}$ and such that the map $\exp_1: \mathcal{V}_1 \rightarrow G$, $(p,v) \mapsto \exp(\natural_1(v))p$ is open. Such a set $\mathcal{V}_1$ always exists since $\exp_1$ is a submersion in a neighbourhood of $U \times \{0\}$. Then set 
\begin{equation}\label{eq U Exp ouverte}
\mathcal{V}  = \{(p,t, v) \in U \times \R_+^* \times \R^I \; |\; (p, \tilde \alpha_t(v)) \in \mathcal{V}_1 \} \sqcup \beta^{-1}(U \times \{0\}) \times \R^I
\end{equation}

The set $\mathcal{V}$ is easily seen to be open; it remains to prove that $\Exp_{|\mathcal{V}}$ is an open map. Let $W \subset \mathcal{V}$ be an open set; by \Cref{lemme ouvert independant base}, it suffices to prove that $\Exp^{-1}(\Exp(W))$ is open, using only the family $\mathbb{X}$. The set $\Exp^{-1}(\Exp(W)) \cap ( \units G \times \R_+^* \times \R^I)$ is open by construction of $\mathcal{V}$. It remains to show that, for any point of the form $(p_0, L_0, 0, v_0) \in W$ and any other vector $v_0' \in \R^I$ such that
\begin{equation}\label{eq egalite exponentielles}
e^{\natural_{p,0}(v_0)} = e^{\natural_{p,0}(v'_0)} \mod L_0,
\end{equation}
there is an open neighbourhood of $(p_0, L_0, 0, v_0')$ contained in $\Exp^{-1}(\Exp(W))$. We claim that there exists $\mathcal{D} \subseteq U  \times \Grass(\R^I) \times \R_+$ an open neighbourhood of $U \times \{0\} \times \Grass(\R^I)$ and a commutative diagram 

\begin{equation}\label{diagramme Exp ouverts}
\begin{tikzcd}[column sep = tiny]
\mathcal{D} \times \R^I \ar[rr, "\tilde \Phi_0"]  && U \times \R_+ \times \Grass(\R^I) \times \R^I  \\
\iota^{-1}(\mathcal{D})\times \R^I \ar[rr, "\Phi_0"] \ar[rd, "\Exp"'] \ar[u, "\iota"]&& \beta^{-1}(U \times \R_+) \times \R^I \ar[ld, "\Exp"] \ar[u, "\iota"'] \\
& \deformation{G}{\module F}&
\end{tikzcd}
\end{equation}
with $\Phi_0(p_0, L_0, 0, v_0') = (p_0, L_0, 0, v_0)$; the set $\Phi_0^{-1}(W)$ will be the desired open neighbourhood of $(p_0, L_0, 0, v_0')$.

Let $w_0 \in \R^I$ be such that $\phi_{p_0, 0}(v_0, w_0) = v_0'$; using \eqref{eq phi composition t nul} and \eqref{eq egalite exponentielles}, one computes 
\begin{equation}
e^{\natural_{p_0, 0}(v_0)} e^{\natural_{p_0, 0}(w_0)} = e^{\natural_{p_0, 0}(v_0')} = e^{\natural_{p_0, 0}(v_0)} \mod L_0
\end{equation}
hence $\natural_{p_0, 0}(w_0) \in L_0$. Let $\Taut \rightarrow \Grass(\R^I)$ be the tautological bundle, ie the vector bundle whose fiber over $L$ is $L$. Fix $L \mapsto w(L) \in L$ a continuous section of $\Taut$ such that $w(\natural_{p,0}^{-1}(L_0)) = w_0$ and set 
\begin{equation}
\mathcal{D} = \{ (p,L,t) \in U \times \R_+ \times \Grass(\R^I) \; |\; (p,t, w(L)) \in \mathcal{U} \}.
\end{equation}
Finally set $\tilde \Phi_0(p,t,L, u) = (p, t,L, \phi_{p,t}(u, w(L)))$; the diagram \eqref{diagramme Exp ouverts} then commutes by construction.

\item \emph{Hausdorff property:} Using that $G \times \R_+^*$ and $\Grass(\R^I)$ are Hausdorff and that the canonical map $\deformation{G}{\module F} \rightarrow \R_+$ is continuous, the only pairs of points that are not trivially separated are $((p_0, g_1 \mod e^{L_0}, 0), (p_0, g_2 \mod e^{L_0}, 0)) \in (\Osc_{\module F}(G) \times \{0\})^2$, with $g_1 \neq g_2 \mod e^{L_0}$. Consider such a pair, assume that $p_0 \in U$ and fix $v_1, v_2 \in \R^I$ such that $g_i = e^{\natural_{p_0,0}(v_i)}$, $i = 1,2$. We want to build open sets $(p_0,L_0,0, v_i) \in \mathcal{V}_i \subseteq \beta^{-1}(U \times \R_+) \times \R^I$, $i = 1,2$, such that $\Exp(\mathcal{V}_1) \cap \Exp(\mathcal{V}_2) = \emptyset$. Up to intersect the sets $\mathcal{V}_i$'s with the open set $\mathcal{V}$ of \eqref{eq U Exp ouverte}, the sets $\Exp(\mathcal{V}_i)$ will be open and separate the points $(p_0, g_i \mod e^{L_0}, 0)$, $i= 1,2$.

Consider the diagram 
\begin{equation}
\begin{tikzcd}
U \times \Grass(\R^I) \times \R_+ \times \R^I \times \R^I \ar[r, "\tilde \Pi"] & U \times \Grass(\R^I) \times \R_+ \times \R^I \ar[d] \\
\beta^{-1}(U \times \R_+) \times \R^I \times \R^I \ar[u, "\iota"] \ar[r, "\Pi"] & U \times \R_+ \times \CoTaut
\end{tikzcd}
\end{equation}
where $\CoTaut \rightarrow \Grass(\R^I)$ denotes the co-tautological bundle, ie the vector bundle whose fiber over $L$ is $\R^I / L$, the right vertical arrow is the canonical projection from $\Grass(\R^I) \times \R^I$ to $\CoTaut$ and
\begin{equation}
\tilde \Pi (p, L, t, u, v) = (p, L, t, \phi_{p,t}(-u, v)).
\end{equation}

Let $W \subseteq \R^I$ be an open neighbourhood of $0$ as given by \Cref{periodic bounding} and set 
\begin{align}
\mathcal{W}_0 & = \left\{ (p,L,t, v) \in U \times \Grass(\R^I)\times \R^I \; |\; \tilde \alpha_t(v) \in W \right\} \\
\mathcal{W}& = (\tilde \Pi \circ \iota)^{-1}(\mathcal{W}_0).
\end{align}

By construction of $W$, for any $(p, t, v) \in U \times \R_+^* \times \R^I$ such that $\tilde \alpha_t(v) \in W$, one has $\exp(\natural_t(v)) p = p \Leftrightarrow v \in \Ker(\natural_{p,t})$. It thus follows from \eqref{eq phi composition t non nul} and \eqref{eq phi composition t nul} that, for any $(a,u,v) \in \mathcal{W}$, $\Exp(a,u) = \Exp(a,v)$ if and only if $\Pi(a,u,v)$ belongs to the zero section of $\CoTaut$. Set
\begin{equation}
\mathcal{D} = \left\{(a,u,v) \in \mathcal{W} \; |\; \Pi(a,u,v) \in U \times \R_+ \times (\CoTaut \setminus \{0\}) \right\}.
\end{equation}

The set $\mathcal{D}$ is open by construction and contains $(p_0, L_0, 0, v_1, v_2)$, hence one can find open sets $(p_0, L_0, 0, v_i) \in \mathcal{V}_i \subseteq \beta^{-1}(U \times \R_+) \times \R^I$, $i= 1,2$, such that $\{(a,u,v) \; |\; (a,u) \in \mathcal{V}_1\; \text{and} \; (a,v) \in \mathcal{V}_2\}$ is contained in $\mathcal{D}$; such open sets satisfy the required conditions by construction of $\mathcal{D}$.

\item \emph{Continuity of the structure maps:} 
\begin{itemize}
\item The continuity of the inverse map $i$ follows by lifting $i$ as $\tilde i: \beta^{-1}(U \times \R_+) \times \R^I \rightarrow \beta^{-1}(U \times \R_+) \times \R^I $ by $\tilde i (a, u) = (\tilde r(a,u), -u)$. The map $\tilde i$ lifts $i$ because of \eqref{eq multiplication groupoide flots}.

\item For the continuity of the source and range maps, lift them as
\begin{equation}
\begin{tikzcd}[column sep = tiny]
\beta^{-1}(U \times \R_+) \times \R^I \ar[rr, "\Exp"] \ar[rd, shift right, "\tilde r"'] \ar[rd, "\tilde s"]& & \deformation{G}{\module F} \ar[ld, "r"] \ar[ld,  shift right, "s"']\\
&\deformation{\units G}{\module F} & 
\end{tikzcd}
\end{equation}
where $\tilde s$ is the first projection and $\tilde r(a, \sum_i \lambda_i e_i) = \exp_\rho(\sum_i \lambda_i \theta_{k(i)}(X_i)) a$, $\theta_k$ being defined by \eqref{eq theta champs de vecteurs} and $\exp_\rho$ by \Cref{extension de l'exponentielle}. The continuity of $\tilde r$ is a consequence of the construction of $\exp_\rho$, see the proof of \Cref{extension de l'exponentielle}.

The source map is open by \Cref{lemme ouvert independant base}, since $\tilde s^{-1}(\tilde s(\mathcal{V}))$ is open for any open set $\mathcal{V} \subseteq \beta^{-1}(U \times \R_+) \times \R^I$, hence so does the range map because of the openness of the inverse and the relation $r = s \circ i$.

\item Since the lift $\tilde s$ of the source map is just the first projection, one identifies
\begin{equation}
\begin{aligned}
&\left( \beta^{-1}(U \times \R_+) \times \R^I \right) \underset{\tilde s \; \tilde r}{\times } \left( \beta^{-1}(U \times \R_+) \times \R^I \right) \\
 \simeq & \left\{ (a,u, v) \in \beta^{-1}(U \times \R_+) \times \R^I \times \R^I \; |\; \tilde r (a,u) \in \beta^{-1}(U \times \R_+) \right\}
\end{aligned}
\end{equation}
Denote by $\mathcal{P}$ this open set. For the continuity of the multiplication map $m$, it suffices to lift $m$ as
\begin{equation}
\begin{tikzcd}
\mathcal{P}' \ar[d, "\Exp \underset{\tilde s \; \tilde r}{\times } \Exp"'] \ar[r, "\tilde m"]& \beta^{-1}(U \times \R_+) \times \R^I \ar[d, "\Exp"] \\
\deformation{G}{\module F}^{(2)} \ar[r, "m"] & \deformation{G}{\module F}
\end{tikzcd}
\end{equation}
where $\tilde m (a, u, v) = \phi_{\beta(a)}(u,v)$ , $\mathcal{P}' = \{ (a, u, v) \in  \mathcal{P} \; | \; (\beta(a), v) \in \mathcal{U} \}$ and $\Exp \underset{\tilde s \; \tilde r}{\times } \Exp (a, u, v) = (\Exp( \tilde r( a, v), u), \Exp(a,v))$; the diagram commutes by construction, using \eqref{eq multiplication groupoide flots}. Note that one should a priori lift $m$ to any fibered product between any pair of generating families, but \Cref{lemme ouvert independant base} allows to restricts to the above case, since any composable pair $(\gamma, \gamma') \in (\Osc_{\module F}(G) \times \{0\})^{(2)}$ with $s(\gamma') \in \beta^{-1}(U \times \{0\})$ is of the form $\Exp \underset{\tilde s \; \tilde r}{\times } \Exp (a,u,v)$ with $(a,u,v) \in \mathcal{P}'$.
\end{itemize}

\end{itemize}

\end{proof}

\subsection{Smooth structure}\label{section structure lisse}

\emph{\textbf{As for $\deformation{\units G}{\module F}$, there is no canonical structure of smooth manifold on $\deformation{G}{\module F}$.}} Nevertheless, we can still define a class of "smooth functions" as for the unit space, as follows.

\begin{definition}\label{def fonctions lisses groupoide}
Let $f$ be a continuous function on $\deformation{G}{\module F}$. We say that $f$ is \textbf{smooth} if:
\begin{enuminthm}
\item\label{fonction lisse t non nul} $f_{|G \times \R_+^*}$ is smooth;
\item\label{pullback exponentielle lisse} the map $f \circ \Exp$ belongs to $\ci(\beta^{-1}(U \times \R_+)\times \R^I)$ (in the sense of \Cref{def fonctions lisses unites}) for any generating family $(U, (X_i)_{i \in I}, k)$.
\end{enuminthm}

We denote by $\ci(\deformation{G}{\module F})$ the set of (complex valued) smooth functions. More generally, for $N$ a smooth manifold, we define $\ci(\deformation{M}{\module F} \times N)$ by modifying \ref{fonction lisse t non nul} and \ref{pullback exponentielle lisse} in the obvious way.
\end{definition}

Using smoothness of the map $\Psi$ in the proof of \Cref{lemme ouvert independant base}, one shows the following.

\begin{lemme}
Let $f$ be a continuous (complex valued) function on $\deformation{G}{\module F}$. If:
\begin{enuminthm}
\item $f_{|G \times \R_+^*}$ is smooth,
\item\label{pullback exponentielle lisse} for every point $p \in M$, there is a generating family $\mathbb X = (U, (X_i)_{i \in I}, k)$ such that $U$ contains $p$ and the map $f \circ \Exp^{\mathbb X}$ belongs to $\ci(\beta^{-1}(U \times \R_+)\times \R^I)$,
\end{enuminthm}
then $f$ is smooth.
\end{lemme}

\begin{remarque}
It is left to the reader to check that the algebra $\ci(\deformation{G}{\module F})$ defines a quasi-Lie structure on $\deformation{G}{\module F} \rightrightarrows \deformation{\units G}{\module F}$, in the sense of \cite[Section 1.2]{Mohsen26}.
\end{remarque}

The following proposition is straightforward by construction of $\exp(Y)$.

\begin{proposition}
Let $Y \in \cci(\deformation{\units G}{\module F}, \deformation{\bundle A G}{\module F})$ be of the form of \Cref{prop flot s fibres}.
\begin{enuminthm}
\item The homeomorphism $\exp(Y)$ is smooth, in the sense that a map $f: \deformation{M}{\module F} \rightarrow \C$ is smooth if and only if $f \circ \exp_\rho(Y) $ is smooth. 
\item The section $Y$ induces a well defined linear map $Y: \ci(\deformation{G}{\module F}) \rightarrow \ci(\deformation{G}{\module F})$ by the formula
\begin{equation}\label{derivation fonctions algebroide}
(Y \cdot f)(\gamma) \coloneqq \left.\frac{d}{ds}\right|_{s=0} f(\exp(sY) \gamma).
\end{equation}
\end{enuminthm}
\end{proposition}

\begin{remarque}
Replacing smooth functions by smooth half densities, one easily shows that \eqref{derivation fonctions algebroide} still holds and that $Y\cdot (f*g) = (Y \cdot f)*g$. Moreover, denoting $Y* f \coloneqq Y \cdot f$, one defines similarly a right convolution $f * Y$ and show that $(f * Y) * g = f* (Y *g)$; $Y$ thus defines a r,s-distribution, in the sense of \cite[Section 1.6]{Mohsen26}. We aim to study these properties in a forthcoming article.
\end{remarque}

\subsection{Debord-Skandalis action}
Recall that, for all $p \in M$, there is a canonical action of $\R_+^*$ on $\gr(\module F)_p$ denoted by $\alpha$ and defined by \eqref{action de R sur la localisation}. This action is easily seen to be a Lie algebra homomorphism, hence it exponentiates as an action by group homomorphisms on $\Gr(\module F)_p$. The group $\R_+^*$ also canonically acts smoothly on $\deformation{\units G}{\module F}$ by \Cref{def action debord skandalis}. 

\begin{definition}
We still denote by $\alpha$ the action $\R_+^* \curvearrowright \deformation{G}{\module F}$ defined by 
\begin{equation}
\begin{aligned}
\alpha_\lambda(\gamma,t) &= (\gamma, \lambda^{-1} t) && \text{for} \; (\gamma, t) \in G \times \R_+^* \\
\alpha_\lambda(p, g \mod e^L, 0) &= (p, \alpha_\lambda(g) \mod e^{\alpha_\lambda(L)}, 0) && \text{for} \; (p, g \mod e^L) \in \Osc_{\module F}(G)
\end{aligned}
\end{equation}
for all $\lambda \in \R_+^*$. The action $\alpha$ is still called the \textbf{Debord-Skandalis action}.
\end{definition}

\begin{lemme}
The action $\alpha$ is continuous. Moreover it is smooth in the sense that, for all $f \in \ci(\deformation{G}{\module F})$, the map $(g, \lambda) \mapsto f(\alpha_\lambda(g))$ belongs to $\ci(\deformation{G}{\module F} \times \R_+^*)$.
\end{lemme}

\begin{proof}
Let $(U, (X_i)_{i \in I}, k)$ be a generating family and define $\tilde \alpha$ the linear action of $\R_+^*$ on $\R^I$ given by $\tilde \alpha_\lambda(e_i) = \lambda^{k(i)} e_i$. We still denote $\tilde \alpha$ the smooth action of $\R_+^*$ on $\beta^{-1}(U \times \R_+) \times \R^I$, given by $\tilde \alpha_\lambda (a, v) = (\alpha_\lambda(a), \tilde \alpha_\lambda(v))$. Then $\tilde \alpha$ lifts $\alpha$ through $\Exp$, hence $\alpha$ is continuous and smooth.
\end{proof}

\section{Examples}\label{section exemples}
\subsection{Equiregular case}\label{ex equiregulier}
Let $\bundle E \rightarrow M$ be a smooth vector bundle. Given an increasing family of subbundles $\{0\} = \bundle F^0 \subseteq \bundle F^1 \subseteq \cdots \subseteq \bundle F^N = \bundle E$, one can define a filtration of depth $N$ by
\begin{equation}
\module F^k = \ci(M, \bundle F^k),\quad k = 0, \dots, N.
\end{equation}
We call \textbf{equiregular} such a filtration. In this case the evalutation maps $\module F^k \rightarrow \bundle F^k_p$,  $X \mapsto X(p)$, where $\bundle F^k_p$ denotes the fiber of $\bundle F^k$ at $p$, induce an isomorphism
\begin{equation}
\gr(\module F)_p \simeq \bigoplus_{k = 1}^N \frac{\bundle F^k_p}{\bundle F^{k-1}_p}.
\end{equation}
In particular, the dimension of $\gr(\module F)_p$ is locally constant (equal to $\dim(\bundle E_p)$). Moreover, a triplet $(U, (X_i)_{i \in I}, k)$ is a generating family if and only if $\{X_i \; |\; k(i) \leq l \}$ generates the bundle $\bundle F^l_{|U}$ for all $l = 1, \dots, N$. 

Let $p \in U$: up to reduce the open set $U$ around $p$ and to extract a subfamily of $(X_i)_{i \in I}$, one can assume that, for all $ l = 1, \dots, N$, the set $\{X_i \; |\; k(i) \leq l \}$ forms a frame of $\bundle F^l$ over $U$. In this case the maps $\natural_{q,t}$ are isomorphisms for all $q \in U$, hence $\blowup{M}{\module F}_p = \{\{0\}\}$. One thus simply gets
\begin{align}
\blowup{M}{\module F} &= M, \\
\deformation{M}{\module F} &= M \times \R_+, \\
\Lie{osc}_{\module F}(\bundle E) &= \bigsqcup_{p \in M} \gr(\module F)_p
\end{align}
and the topology given by \Cref{def topo blup} is just the product one on $M \times \R_+$. Furthermore, the smooth functions defined by \Cref{def fonctions lisses unites} are simply the smooth functions for the structure of product manifold (with boundary) on $M \times \R_+$. The "smooth structure" on the deformation bundle $\deformation{\bundle E}{\module F} \rightarrow M \times \R_+$ is also a structure of smooth vector bundle in the usual sense (over a manifold with boundaries) and the smooth sections in the sense of \Cref{def structure lisse fibre} are exactly the smooth sections in the usual sense.

In the case where $\bundle E = \bundle T M$ and $\module F$ is a Lie filtration for the canonical Lie algebroid structure $(\bundle T M, [ \cdot , \cdot ], \id)$, then $\deformation{M \times M}{\module F} \rightrightarrows M \times \R_+$ is the groupoid $\mathbb{T}_{\bundle F} M$ defined in \cite{vanErpYuncken16}. More generally, if $G$ is a Lie groupoid and $\module F$ an equiregular Lie filtration on $(\bundle A G, [ \cdot, \cdot], \rho)$, then $\deformation{G}{\module F} \rightrightarrows \units G \times \R_+$ is the groupoid $\mathbb{A}_{\bundle F}G$ defined in \cite[Section 9]{vanErpYuncken16}.

\subsection{Formal Baouendi-Grushin filtration}

Let $N \geq 2$, $M = \R^2$ and $\bundle E \rightarrow M$ a trivial bundle of rank $2$, generated by a global basis of sections denoted by $(A,B)$. One can generalize the Baouendi Grushin filtration of \Cref{exemple Baouendi classique} by defining the filtration
\begin{equation}
\module F^k = \left\{f(x,y) A + g(x,y) y^{N -k}B  \; | \; f, g \in \ci(M, \R) \right\} ,\quad k = 1, \dots, N.
\end{equation}
All the computations of \Cref{exemple Baouendi classique} and \Cref{exemple Baouendi blup} can be adapted to this filtration, replacing $(\partial_y, \partial_x)$ by $(A,B)$ everywhere. Let us summarize them. Let $p =(x_0, y_0) \in M$:

\begin{align}
\gr(\module F)_p &= \begin{cases}
\Span([A]_{p, 1}, [B]_{p, 1}) & \text{if} \; y_0 \neq 0\\
\Span([A]_{p, 1}, [y^{N-1}B]_{p, 1}, [y^{N-2}B]_{p, 2}, \dots, [B]_{p, N}) & \text{if} \; y_0 = 0,
\end{cases}\\
\blowup{M}{\module F}_p  &= \begin{cases}
\{\{0\}\} & \text{if} \; y_0 \neq 0\\
\{L_\theta^p \; |\; \theta \in \mathbb{P}^1(\R) \} & \text{if} \; y_0 = 0
\end{cases}
\end{align}
with, for $y_0 = 0$:
\begin{equation}
L_{[a,b]}^p = \Ker\left(\sum_{i=1}^N a^{N-i}b^{i-1} [y^{N-i}B]_{p,i}^* \right) \cap \Ker \left( [A]_{p,1}^*\right) \in \Grass(\gr(\module F)_p)
\end{equation}
where $([A]_{p,1}^*, [y^{N-1}B]_{p,1}^*, \dots, [B]_{p,N}^* )$) denotes the dual basis of $\gr(\module F)_p^*$ and $[a,b] \in (\R^2 \setminus \{(0,0)\})/\R^* = \mathbb{P}^1(\R)$. Moreover, the topology of $\deformation{M}{\module F}$ is described by the following topological embedding:
\begin{equation}\label{plongement topo def projectif}
\begin{aligned}
\deformation{M}{\module F} & \hookrightarrow M \times \mathbb{P}^1(\R) \times \R_+ &&  \\
((x,y),t) & \mapsto ((x,y), [y,t], t) && \text{if} \; t \neq 0 \\
((x,y), \{0\}, 0) & \mapsto ((x,y), [1,0], 0) &&\text{if} \; y \neq 0 \\
((x,0), L_\theta^{(x,0)}, 0) & \mapsto ((x,0), \theta, 0).&& \text{if} \; y = 0,\; t = 0
\end{aligned}
\end{equation}
Finally, for $p \in \R \times \{0\}$ and $L^p_{[a,b]} \in \blowup{M}{\module F}_p$, one computes from \eqref{action de R sur la localisation} that $\alpha_\lambda(L_{[a,b]}^p) = (L_{[a, \lambda^{-1}b]}^p)$. The action of $\R_+^*$ thus has three orbits in $\blowup{M}{\module F}_p$ which are $\{L_{[1,0]}^p\}$, $\{L_{[0,1]}^p\}$ and $\{L_{[a,1]}^p \; |\; a \ \neq 0 \}$.

\subsection{Different Lie algebroid structures}\label{section algebroides differentes Baouendi}

In the case of the formal Baouendi-Grushin filtration, the bundle $\bundle E$ can be endowed with different Lie algebroid structures for which $\module F$ is a Lie algebroid filtration. Let us describe three examples and detail for each of them the action by conjugation $\Gr(\module F)_p \curvearrowright \blowup{M}{\module F}_p$, using \Cref{blup stable par conjugaison}. As for \Cref{ancre paires} let $p = (x_0, 0) \in \R \times \{0\}$ (it is the only interesting case), $\theta = [a,b] \in \mathbb{P}^1(\R)$ and $v = \sum_i b_i [y^{N-i}B]_{p,i} + c[A]_{p,1} \in \gr(\module F)_p$; we will compute $e^v L_\theta^p e^{-v}$. 

For this purpose, consider $(p_n, t_n) = ((x_n, y_n), t_n) \in M \times \R_+^*$ a sequence that converges to $ (p, L_\theta^p, 0) \in \blowup{M}{\module F} \times \{0\}$ (for the topology of $\deformation{M}{\module F}$). By \eqref{plongement topo def projectif} it is equivalent to say that $(p_n,t_n, [y_n, t_n]) \rightarrow (p,0, \theta)$ for the topology of $M \times \R_+ \times \mathbb{P}^1(\R)$. Finally, let $f$ be a compactly supported function on $M$ with value $1$ in a neighbourhood of $p$ and consider the section of $\deformation{\bundle E}{\module F}$: $X = \sum_i b_i \theta_i(f y^{N-i}B) + c\theta_1(f A)$ ($\theta_i$ being defined by \eqref{eq theta champs de vecteurs}). Nothing depends on a choice of Lie algebroid structure on $\bundle E$ so far.

\begin{enumerate}[label =\roman*)]
\item\label{algebroide paires ex baouendi} We can take the anchor $\rho$ to be the isomorphism $\rho(A) = \partial_y$ and $\rho(B) = \partial_x$. The Lie bracket is then defined by $\rho([X,Y]) = [\rho(X), \rho(Y)]$. In that case, the Lie structure on $\gr(\module F)_p$ and the action $\Gr(\module F)_p \curvearrowright \blowup{M}{\module F}_p$ is the same as in \Cref{ancre paires}, namely the Lie bracket is
\begin{equation}
\begin{aligned}
[[A]_{p,1}, [y^{N-i}B]_{p,i}] &= (N-i)[y^{N-i-1}B]_{p,i+1} && \forall i = 1, \dots, N-1\\
[[A]_{p,1}, [B]_{p,N}] &= [[y^{N-i}B]_{p,i}, [y^{N-j}B]_{p,j}] = 0 && \forall i,j = 1, \dots, N.
\end{aligned}
\end{equation}
and the action is
\begin{equation}
e^vL_{[a,b]}^p e^{-v} = L_{[a+bc, b]}^p.
\end{equation}

\item\label{algebroide b ex baouendi} We can take the anchor $\rho^b$ defined by $\rho^b(A) = \partial_y$ and $\rho^b(B) = x\partial_x$. The map $\rho^b$ is not fiberwise injective but $\rho^b: \cci(M, \bundle E) \rightarrow \cci(M, \bundle T M)$ is injective. Its image is equal to $\{X \in \cci(M, \bundle TM) \; |\; X(0, y) \in \R \partial_y \; \forall y \in \R \}$ and is thus stable under bracket, therefore the relation $\rho^b([X,Y]) = [\rho^b(X), \rho^b(Y)]$ defines a Lie bracket on $\cci(M, \bundle E)$.

One easily checks that the Lie algebra structure on $\gr(\module F)_p$ is the same as the one of \ref{algebroide paires ex baouendi}. Moreover, let $(p_n', t_n) = \exp_{\rho^b}(X)(p_n, t_n)$ with $p_n' = (x_n', y_n')$. We still have $y_n' = y_n + t_n c$, hence $(p_n', t_n) \rightarrow (p, L_{[a + bc,b]}^p, 0)$ in $\deformation{M}{\module F}$: the action by conjugation is also the same as in \ref{algebroide paires ex baouendi}.

\item\label{algebroide edge ex baouendi} We can take the anchor $\rho^e$ defined by $\rho^e(A) = x \partial_y$ and $\rho^e(B) = x\partial_x$. The map $\rho^e$ is also bijective between sections (with image $\{X \in \cci(M, \bundle TM) \; |\; X(0, y) = 0 \; \forall y \in \R \}$) and the bracket is defined the same way as in \ref{algebroide b ex baouendi}. The bracket on $\gr(\module F)_p$ is still vanishing for $p \in \R \times \R^*$, and for $p = (x_0,0)$ one computes:

\begin{equation}
\begin{aligned}
[[A]_{p,1}, [B_i]_{p,i}] &= x_0 (N-i)[B_{i+1}]_{p,i+1} && \forall i = 1, \dots, N-1\\
[[A]_{p,1}, [B_N]_{p,N}] &= [[B_i]_{p,i}, [B_j]_{p,j}] = 0 && \forall i,j = 1, \dots, N.
\end{aligned}
\end{equation}

Note in particular that $\Gr(\module F)_{(0,0)}$ is abelian. Set $(p_n', t_n) = \exp_{\rho^e}(X)(p_n, t_n)$ with $p_n' = (x_n', y_n')$. There is no easy formula for $y_n'$ since it requires to solve a non linear ODE. However, one can exhibit, using Gronwall lemma, a constant $K > 0$ such that $|y_n' - (y_n + t_n c x_n)| \leq K t_n^2 x_n$ hence $(p_n', t_n) \rightarrow (p, L^p_{[a + bcx_0, b]}, 0)$ in $\deformation{M}{\module F}$. The action by conjugation thus becomes
\begin{equation}
e^vL_{[a,b]}^p e^{-v} = L_{[a+bcx_0, b]}^p.
\end{equation}
\end{enumerate}

The structure of \ref{algebroide paires ex baouendi} corresponds to the generalized Baouendi-Grushin filtration on the tangent bundle, see \cite[Example 1.9]{Mohsen24}, while \ref{algebroide b ex baouendi} and \ref{algebroide edge ex baouendi} are variations of this example respectively on the b-tangent bundle and the 0-tangent bundle, see \Cref{section b calcul} and \Cref{section 0 calcul}.

\subsection{Reminders on manifolds with boundary}\label{section rappels bord}

Let $(M, \partial M)$ be a manifold with boundary. To study the ellipticity of certain differential operators with boundary conditions, Melrose introduced a Lie algebroid $(\bundle T^b M \rightarrow M,[\cdot, \cdot], \rho^b)$ called the \textbf{b tangent bundle}, see \cite{Melrose93}, satisfying 
\begin{equation}
\rho^b(\ci(M, \bundle T^bM)) = \{ X \in \ci(M, \bundle TM) \; |\; X_{|\partial M} \in \bundle T \partial M \}.
\end{equation}
In \cite{Monthubert} the author showed that $(\bundle T^b M ,[\cdot, \cdot], \rho^b)$ can be realised as the Lie algebroid of a certain Lie groupoid over $M$. We will denote this groupoid $(M\times M)^b \rightrightarrows M$ and call it the \textbf{Monthubert groupoid}. More generally, given any submersion $p: \partial M \rightarrow B$, one can build a Lie groupoid over $M$ with Lie algebroid $(\bundle T^eM, [\cdot, \cdot], \rho^e)$ called the edge tangent bundle, satisfying 
\begin{equation}
\rho^e(\ci(M, \bundle T^e M)) = \{ X \in \ci(M, \bundle TM) \; |\; X_{|\partial M} \in \Ker dp \}.
\end{equation}
In general this Lie groupoid is $(M \times M)^e = \Blup_{r,s}(M \times M, \partial M {}_p \! \times \! {}_p \partial M)$, see \cite{DebordSkandalis17} for a modern description of blowup groupoids and \cite{Rochon12} for applications to analysis. 

The b-tangent bundle corresponds to the case $B = \{ pt\}$. Another interesting example is the case $B = \partial M$ and $p = \id$, corresponding to the so called uniformly degenerate vector fields, see \cite{MazzeoMelrose98}. As explained in the introduction, many results of Melrose on manifolds with boundary can be recovered using pseudodifferential calculus on groupoids.

To avoid dealing with the boundary one can replace $M$ by $\tilde M = M_- \cup M_+$, the manifold (without boundary) obtained by gluing two copies $M_\pm$ of $M$ over $\partial M$. The boundary $\partial M$ thus becomes a submanifold of $\tilde M$ of codimension $1$. All the above constructions can be applied replacing $(M, \partial M)$ by a pair $(M, N)$ where $M$ is a manifold without boundary and $N \subset M$ a submanifold of codimension $1$; it will be our setting in the following.

\subsection{Baouendi-Grushin for the b-calculus}\label{section b calcul}

Choose a smooth function $x: M \rightarrow \R$, with non vanishing differential on $N$, such that $N = x^{-1}(0)$\footnote{It can always be built using a tubular neighbourhood.}. The function $x$ is called a \textbf{defining function} for $N$. Using $x$ one can find an open neighbourhood of $N$ in $ M$ which is diffeomorphic to $\R \times N$. To simplify the description of $(M \times M)^b$ we will assume that $M = \R \times N$, the submanifold of interest being $\{0\} \times N \simeq N$. Denoting by $x$ the $\R$ variable and by $y$ the $N$ variable, the groupoid $(M \times M)^b \rightrightarrows M$ is
\begin{itemize}
\item[•] $(M \times M)^b =  \R \times \R^* \times N \times N$ and $M \hookrightarrow (M \times M)^b,\; (x,y) \mapsto (x,1, y, y)$;
\item[•] $s(x,\lambda, y', y) = (x,y)$ and $r(x,\lambda, y', y) = (y', \lambda x)$;
\item[•] $(\lambda x,\mu, y'', y') \cdot (x,\lambda, y', y) = (x, \lambda \mu, y'',y)$;
\item[•] $(x,\lambda, y', y)^{-1} = (\lambda x, \lambda^{-1}, y, y') $.
\end{itemize}

\begin{remarque}\label{Monthubert est canonique}
In other words, when $M = \R \times N$, $(M \times M)^b$ is a direct product between the action groupoid $\R \rtimes \R^*$ and the pair groupoid $N \times N$. If $(M,N)$ is any pair with $\codim(N) = 1$, the groupoid $(M \times M)^b$ is isomorphic to a gluing between $(\R \rtimes \R^*) \times (N \times N)$ and the pair groupoid $(M \setminus N) \times (M \setminus N)$, the gluing depending on the defining function $x$. Nevertheless, as explained above there is a general definition $(M \times M)^b = \Blup_{r,s}(M \times M, N \times N)$ which does not depend on the choice of a defining function, see \cite{DebordSkandalis17}.
\end{remarque}

Denote by $\lambda$ the $\R^*$ coordinate: the vector field $\lambda \partial_\lambda \in \ci((M \times M)^b, \Ker(ds))$ is right equivariant, denote by $\partial_\lambda \in \ci(M, \bundle T^b M)$ its restriction to the units. One thus identifies $\bundle T^b M \simeq \bundle TN \oplus \R \cdot \partial_\lambda$ and one easily checks that, under this identification, the anchor map $\rho^b: \bundle T^b M \rightarrow \bundle T M = \bundle T N \oplus \R \cdot \partial_x$ is the identity on $\bundle T N$ and sends $\partial_\lambda$ to $x \partial_x$.

Let $N = \R$, $K > 0$ and $\module F^k = \left\{f(x,y) y^{K -k}\partial_\lambda + g(x,y) \partial_y \; | \; f, g \in \ci(M, \R) \right\}$, $k = 1, \dots, K$. This filtration is the one described in \Cref{section algebroides differentes Baouendi} in the case \ref{algebroide b ex baouendi}, setting $A = \partial_y$ and $B = \partial_\lambda$.

\subsection{Baouendi-Grushin for the 0-calculus}\label{section 0 calcul}
As in the previous section, assume $M = \R \times N$ and denote $x$ the $\R$ variable and $y$ the $N$ variable. Moreover assume $N = \R^d$ and identify $\bundle T N = N \times \R^d$. Let us describe the edge groupoid $(M \times M)^e \rightrightarrows M$ associated to the constant map $p: N \rightarrow \{pt\}$, as described in the beginning of \Cref{section rappels bord}. 

\begin{itemize}
\item[•] $(M \times M)^e =  \R \times \R^* \times \bundle T N $ and $M \hookrightarrow (M \times M)^b,\; (x,y) \mapsto (x, 1, y, 0)$;
\item[•] $s(x, \lambda, y, \xi) = (x,y)$ and $r(x, \lambda, y, \xi) = (\lambda x, y + x \xi)$;
\item[•] $(\lambda x, \mu, y + x \xi, \eta) \cdot (x,\lambda, y, \xi) = (x, \lambda \mu, y , \xi +  \lambda \eta)$;
\item[•] $(x,\lambda, y, \xi)^{-1} = (\lambda x, \lambda^{-1}, y + x \xi, - \lambda^{-1} \xi) $.
\end{itemize}

\begin{remarque}
If one chooses a riemannian metric on $N$, one may drop the assumption $N = \R^d$, set $(M \times M)^e = \R \times \R^* \times \bundle T N$ and replace $y + \lambda \xi$ by $\exp_y(\lambda \xi) $. For a general pair $(M, N)$ with $\codim(N)=1$, the groupoid $(M \times M)^e$ (still in the case $p: N \rightarrow \{pt\}$) is isomorphic to a gluing between $\R \times \R^* \times \bundle T N$ with the above structure and the pair groupoid $(M \setminus N) \times (M \setminus N)$, the gluing depending on the defining function $x$. Nevertheless, as in \Cref{Monthubert est canonique}, there is a canonical definition $(M \times M)^e = \Blup_{r,s}(M \times M, N {}_p \! \times \! {}_p N)$, independent of $x$.
\end{remarque}

Denote by $\lambda$ the $\R^*$ coordinate and identify $\bundle T N \simeq N \times \R^d$. As in \Cref{section b calcul} one may then identify $\bundle T^eM \simeq \bundle T N \oplus \R \partial_\lambda$ where $\lambda$ still denotes the $\R^*$ coordinate. The map $\rho^e : \bundle T^e M \rightarrow \bundle T M = \bundle T N \oplus \R \partial_x$ then sends $\partial_{y^i}$ to $x \partial_{y^i}$, $i = 1, \dots, d$, and $\partial_\lambda$ to $x \partial_x$.

Let $\module F^k = \left\{f(x,y) y^{K -k}\partial_\lambda + g(x,y) \partial_y \; | \; f, g \in \ci(M, \R) \right\}$, $k = 1, \dots, K$. This filtration is the one described in \Cref{section algebroides differentes Baouendi} in the case \ref{algebroide edge ex baouendi}, setting $A = \partial_y$ and $B = \partial_\lambda$.

\printbibliography[heading = bibintoc]

\end{document}